\documentclass[11pt, reqno]{amsart}

\usepackage{url}
\usepackage{hyperref}
\usepackage{pdfsync}
\usepackage{amssymb}
\usepackage{amsmath}
\usepackage{amsthm}
\usepackage{stmaryrd}
\usepackage{fullpage}	
\usepackage{amscd}   	
\usepackage[all,cmtip]{xy}	
\usepackage{comment} 	
\usepackage{mathrsfs}      
\usepackage[new]{old-arrows}
\usepackage{graphicx}

\usepackage{enumerate}

\usepackage[alphabetic,backrefs,lite]{amsrefs} 

\usepackage[dvipsnames]{xcolor}

\newcommand{\marg}[1]{}

\newcommand{\note}[1]{} 				

\newcommand{\defi}[1]{\textsf{#1}} 				

\newcommand{\rra}{\rightarrow}

\newcommand{\C}{{\mathbb C}}
\newcommand{\F}{{\mathbb F}}

\newcommand{\Q}{{\mathbb Q}}

\newcommand{\Z}{{\mathbb Z}}

\newcommand{\Qbar}{{\overline{\Q}}}

\newcommand{\pp}{{\mathfrak p}}

\newcommand{\mZ}{\mathbb{Z}}

\newcommand{\mQ}{\mathbb{Q}}
\newcommand{\mF}{\mathbb{F}}
\newcommand{\mC}{\mathbb{C}}

\newcommand{\cdef}[1]{{\color{black}\textsf{#1}}}	
\newcommand{\zZ}[1]{\mathbb{Z}/ #1 \mathbb{Z}}

\def\Q{\mathbb{Q}}
\def\C{\mathbb{C}}

\def\P{\mathbb{P}}

\def\Z{\mathbb{Z}}

\DeclareMathOperator{\im}{im}

\DeclareMathOperator{\Aut}{Aut}
\DeclareMathOperator{\Gal}{Gal}

\DeclareMathOperator{\Cl}{Cl} 
 
\DeclareMathOperator{\Div}{Div} 
 
\DeclareMathOperator{\Spec}{Spec}  
 \DeclareMathOperator{\rank}{rank}

\DeclareMathOperator{\Cot}{Cot}

\newcommand{\tors}{{\operatorname{tors}}}
\newcommand{\red}{{\operatorname{red}}}

\newcommand{\SL}{\operatorname{SL}}

 \DeclareMathOperator{\Sat}{Sat}

\DeclareMathOperator{\GCD}{GCD}

\newcommand{\diamondop}[1]{\langle #1 \rangle}

  \DeclareMathOperator{\cc}{c}
 \DeclareMathOperator{\Prin}{prin}

\newcommand{\Cusp}{\Div^{\cc}}
\newcommand{\Cuspo}{\Div^{0,\cc}}
 
 \newcommand{\ClCusp}{\Cl^{\cc}}
 \newcommand{\PrinCusp}{\Prin^{\cc}}

 \numberwithin{equation}{subsection}
\newtheorem{theorem}[subsection]{Theorem}
\newtheorem{lemma}[subsection]{Lemma}
\newtheorem{corollary}[subsection]{Corollary}
\newtheorem{proposition}[subsection]{Proposition}
\numberwithin{equation}{subsection}
\newtheorem{thmx}{Theorem}

\theoremstyle{definition}
\newtheorem{definition}[subsection]{Definition}

\newtheorem{conjecture}[subsection]{Conjecture}
\newtheorem{example}[subsection]{Example}

\theoremstyle{remark}

\newtheorem{remark}[subsection]{Remark}

\makeatletter
\@namedef{subjclassname@1991}{\emph{2020} Mathematics Subject Classification}
\makeatother

\begin{document}

\title[Sporadic cubic torsion]
{Sporadic cubic torsion}

\author{Maarten Derickx}
\address{Department of Mathematics, Massachusetts Institute of Technology, Cambridge, MA
02139, USA}
\email{maarten@mderickx.nl}

\author{Anastassia Etropolski}
\address{Department of Mathematics, Rice University, 
Houston, TX 77005, USA}
\email{aetropolski@rice.edu}

\author{Mark van Hoeij}
\address{Department of Mathematics, Florida State University,
Tallahassee, FL 32306 USA}
\email{hoeij@math.fsu.edu}

\author{Jackson S. Morrow}
\address{Department of Mathematics , Emory University,
Atlanta, GA 30322 USA}
\email{jmorrow4692@gmail.com}

\author{David Zureick-Brown}
\address{Department of Mathematics, Emory University,
Atlanta, GA 30322 USA}
\email{dzb@mathcs.emory.edu}


\date{\today}
\thanks{}

\begin{abstract}
Let $K$ be a number field, and let $E/K$ be an elliptic curve over $K$. 
The Mordell--Weil theorem asserts that the $K$-rational points $E(K)$ of $E$ form a finitely generated abelian group. 
In this work, we complete the classification of the finite groups which appear as the torsion subgroup of $E(K)$ for $K$ a cubic number field.

To do so, we determine the cubic points on the modular curves $X_1(N)$ for 
\[N = 21, 22, 24, 25, 26, 28, 30, 32, 33, 35, 36, 39, 45, 65, 121.\]
As part of our analysis, we determine the complete list of $N$ for which $J_0(N)$ (resp., $J_1(N)$, resp., $J_1(2,2N)$) has rank 0. 
We also provide evidence to a generalized version of a conjecture of Conrad, Edixhoven, and Stein by proving that the torsion on $J_1(N)(\mQ)$ is generated by $\Gal(\overline{\mQ}/\mQ)$-orbits of cusps of $X_1(N)_{\overline{\mQ}}$ for $N\leq 55$, $N \neq 54$. 
\end{abstract}

\maketitle

\section{Introduction}
\label{sec:introduction}

Let $E/\mQ$ be an elliptic curve defined over the rationals $\mQ$. In 1901, Poincar\'e \cite{poincare1901} conjectured that the set $E(\mQ)$ of $\mQ$-rational points on $E$ is a finitely generated abelian group. Mordell \cite{mordell1922rational} proved this conjecture in 1922, and Weil \cite{weil1929arithmetique} generalized this in 1929 to an arbitrary abelian variety defined over a number field.

\begin{theorem}[Mordell--Weil]
For an elliptic curve defined over a number field $K$,
$$E(K) \cong \mZ^r \oplus E(K)_{\tors}.$$
\end{theorem}
The free rank $r$ of $E(K)$ is the \cdef{rank of $E$ over $K$} and the finite group $E(K)_{\tors}$ is the \cdef{torsion subgroup of $E(K)$}. Since $E(K)_{\tors}$ is isomorphic to a finite subgroup of $\left(\mathbb{Q} / \mathbb{Z}\right)^2$, we know that this group must be isomorphic to a group of the form
$$\zZ{N} \oplus \zZ{NM},$$
for positive integers $N,M$. The celebrated result of Mazur classified which $N,\,M$ appear for $K = \Q$.

\begin{theorem}[Mazur \cite{Mazur:eisenstein}]\label{thm:Mazur}
Let $E/\mQ$ be an elliptic curve. Then $E(\mQ)_{\tors}$ is isomorphic to one of the following 15 groups:
\begin{align*}
&\zZ{N_1} & \text{ with } &1 \leq N_1 \leq 12, N_1 \neq 11,& \\
&\zZ{2} \oplus \zZ{2N_2} & \text{ with } & 1 \leq N_2 \leq 4 .&
\end{align*}
Furthermore, there exist infinitely many $\overline{\mQ}$-isomorphism classes for each such torsion subgroup.
\end{theorem}

The modular curves $X_1(M,MN)$ classify elliptic curves (and degenerations) together with independent points $P,Q$ of order $M$ and $MN$.
In this language, Mazur's theorem asserts that $X_1(M,MN)(\mQ)$ has no non-cuspidal rational points for $(M,MN)$ outside of the above set.

In 1996, Merel \cite{Merel:uniformity} proved the existence of a uniform bound on the size of $E(K)_{\tors}$ that depends only on the degree of the number field $K.$ Merel's result leads to the natural question of classifying (up to isomorphism) the torsion subgroups of elliptic curves defined over number fields of degree $d$, for a fixed integer $d\geq 1$.

For $d = 2$, this classification was started by Kenku--Momose and completed by Kamienny.

\begin{theorem}[Kenku--Momose \cite{kenku1988torsion}; Kamienny \cite{kamienny1992torsion}]
Let $K/\Q$ be a quadratic extension and $E/K$ be an elliptic curve. Then $E(K)_{\tors}$ is isomorphic to one of the following 26 groups:
\begin{align*}
&\zZ{N_1} & \text{ with } & 1 \leq N_1 \leq 18, N_1 \neq 17, \\
&\zZ{2} \oplus \zZ{2N_2} &\text{ with }& 1 \leq N_2 \leq 6 ,\\
&\zZ{3} \oplus \zZ{3N_3} & \text{ with }&N_3 = 1,2, \\
&\zZ{4} \oplus \zZ{4}. & &
\end{align*}
Furthermore, there exist infinitely many $\overline{\mQ}$-isomorphism classes for each such torsion subgroup.
\end{theorem}

The modular curves $X_1(M,MN)$ for the above list all have genus $\leq 2$; all have infinitely many quadratic points.

\subsection{Statement of main results}
In this paper, we shall be concerned with the case of $d= 3$. Jeon, Kim, and Schweizer \cite{Jeon2004Cubic} determined the torsion structures that appear infinitely often as one varies over all elliptic curves over all cubic number fields (i.e., they classified the pairs $(M,MN)$ for which $X_1(M,MN)$ is trigonal, and show there are no $X_1(M,MN)$ which admit a degree 3 map to a positive rank elliptic curve over $\mQ$). 
In \cite{jeon2011familiescubic}, Jeon, Kim, and Lee constructed infinite families of elliptic curves realizing each of these torsion structures by finding models of the relevant trigonal modular curves.
The first author and Najman \cite{derickxNajman:TorsionCyclicCubic} classified the torsion groups of elliptic curves over
cubic fields with Galois group $\mZ/3\mZ$, complex cubic fields, and totally real cubic fields with Galois group $S_3$.

In 2010, Najman \cite{najman2012torsion} discovered the first example of \emph{sporadic} torsion: the elliptic curve $E/\Q$ with Cremona label \href{http://www.lmfdb.org/EllipticCurve/Q/162b1}{\texttt{162b1}} satisfies $E(\mQ(\zeta_9)^+)_{\tors} \cong \zZ{21}$, and is the only elliptic curve defined over $\Q$ which admits a $K$-rational 21-torsion point.
He then classified the possible torsion subgroups of elliptic curves defined over $\mQ$ when considered over some cubic number field $K$. 
The consideration of the base change of elliptic curves defined over $\mQ$ to other number fields leads to sharper results concerning 
the torsion subgroups which appear over number fields $K$ (see \cite{alvaro2013fieldofdefn}). 
A computer generated table of sporadic points is given in \cite{vanHoeij:LowDegreePlaces}. 

Parent \cite{Parent:cubic-torsion-french, Parent:no-17-torsion} proved that cubic torsion points of prime order $p$ do not exist for $p >13$ (reliant on Kato's \cite{Kato:p-adic-hodge-zeta} subsequent generalization of Kolyvagyn's theorem to quotients of $J_1(N)$). Momose \cite[Theorem B]{Momose:quadratic-torsion} ruled out the cyclic torsion for the cases $N_1 = 27, 64$.
 It was formally conjectured in \cite[Conjecture~1.1.2]{wang2015thesis} that the only possible torsion structures for elliptic curves over $K$ are the ones identified by Jeon--Kim--Schweizer and Najman \cite{Jeon2004Cubic, najman2012torsion}. Wang made progress on this conjecture in his thesis (see \cite{wang:cyclictorsion, wang:cyclictorsion2, wang:cyclictorsion3} for updated versions), and ruled out the existence of cyclic torsion for $
N_1 = 77, 91, 143, 169.$
Bruin and Najman \cite[Theorem 7]{bruinN:criterion-to-rule-out-torsion} ruled out the cyclic cases $N_1 = 40, 49, 55$ and the non-cyclic case $N_2 = 10$.

\vspace{2pt}

Our main theorem completes the classification of torsion over cubic number fields.

\begin{thmx}\label{thm:main}
Let $K/\Q$ be a cubic extension and $E/K$ be an elliptic curve. Then $E(K)_{\tors}$ is isomorphic to one of the following 26 groups:
\begin{align*}
&\zZ{N_1} & \text{ with } & N_1 = 1,\dots ,16, 18,20,21, \\
&\zZ{2} \oplus \zZ{2N_2} & \text{ with } & N_2 = 1,\dots , 7.
\end{align*}
There exist infinitely many $\overline{\mQ}$-isomorphism classes for each such torsion subgroup except for $\zZ{21}$. 
In this case, the elliptic curve with Cremona label \href{http://www.lmfdb.org/EllipticCurve/Q/162b1}{\texttt{162b1}} and minimal Weierstrass equation $y^2+xy+y=x^3 - x^2 - 5x+5$ over $\mQ(\zeta_9)^+ \cong \mQ[x]/(x^3-3x+1) $ is the unique elliptic curve over a cubic field with $\zZ{21}$-torsion, in particular the point $(2\alpha - 3 , 2\alpha - 2)$ has order $21$ where $\alpha$ is a root of $x^3-3x+1$. 
\end{thmx}

\begin{remark}[Enumeration of remaining cases]
 \label{remark-enumeration-of-cases}
Combining the above work of Parent, Momose, Wang, and Bruin--Najman, the remaining task\footnote{Wang's proofs for $N_1 = 22,25,39,40,49, 55,65$ contain an error, which we address in Remark \ref{rem:wang2}. (The remaining cases from Wang's papers are correct.)} is to determine the cubic points on the modular curves $X_1(N_1)$ for
\vspace*{-.25em}
\[
N_1 = 21, 22, 24, 25, 26, 28, 30, 32, 33, 35, 36, 39, 45, 65, 121
\]
and on $X_1(2,2N_2)$ for $N_2 = 8,9$.
(Note that there is a natural map $X_1(4N) \to X_1(2,2N)$ by \cite[Section 1]{JeonK:bielliptic-modular-curves}, and thus the determination of the cubic points on $X_1(2,16)$ and $X_1(2,18)$ determines the cubic points on $X_1(32)$ and $X_1(36)$.)
\end{remark}
\vspace*{-.3em}
\subsection{Strategy}
We first determine the complete list of $N$ for which $J_0(N)$ (resp., $J_1(N)$, resp., $J_1(2,2N)$) has rank 0 (Section \ref{sec:computing-ranks}).
For the $N_i$ from Remark \ref{remark-enumeration-of-cases}, $J_1(N_1)(\Q)$ has rank 0, unless $N_1 = 65, 121$.

For the rank 0 cases, we use a variety of techniques:
\begin{itemize}
\item local arguments (\S \ref{subsec:local}),
\item direct computation of preimages of an Abel--Jacobi map $X_1(N)^{(3)}(\mQ) \to J_1(N)(\mQ)$ (\S \ref{subsec:direct}),
\item passage to modular curve quotients (\S \ref{subsec:directanalysisquotient}),
\item results on the cuspidal subgroup of $J_1(N)(\Q)$ (\S \ref{sec:computing-torsion}), and
\item explicit description of cusps on $X_1(N)$ via modular units (\S \ref{subsec:modularunits}).
\end{itemize}
For the rank 1 cases, we use a modified formal immersion criterion (\S \ref{subsec:analysis121} and \S \ref{subsec:analysis65}). These methods are expounded on in Sections \ref{sec:computing-torsion} and \ref{sec:methods}.
\vspace{-.28em}
\subsection{Outline of paper}
In Section \ref{sec:background}, we recall background on the arithmetic of curves, modular curves, cuspidal subschemes of modular curves, and modular units.
In Section \ref{sec:computing-ranks}, we determine the complete list of $N$ for which $J_0(N)$ (resp., $J_1(N)$, resp., $J_1(2,2N)$) has rank 0, and in Section \ref{sec:computing-torsion}, we investigate the torsion subgroup of $J_1(N)(\Q)$ via modular symbols.
We describe the various techniques used to determine the cubic points on $X_1(N)$ in Section \ref{sec:methods}, and we conclude with the determination of cubic points on modular curves in Sections \ref{sec:cubicpoints0} and \ref{sec:cubicpointspos}.
\vspace{-.28em}
\subsection{Conventions}
Let $K$ be a field and let $X/K$ be a {nice} curve i.e., a smooth, proper, geometrically integral scheme of dimension one.
For such an $X$, let $K(X)$ denote its function field and let $J_X$ denote its Jacobian.
For a field $L \supset K$, let $X_L$ denote the base change of $X$ to $L$.
In some cases, we will use this notation when the curve $X$ is defined over a ring $R$.
Let $\Div X$ be the group of all divisors of $X$, and let $\Div^0 X$ be the subgroup of divisors of degree 0.
Let $\Div_L X$ (resp.~$\Div^0_L X$) be the group of all divisors (resp.~the subgroup of divisors of degree 0) of the curve $X_{L}$.
We note that base change gives inclusions $\Div X \hookrightarrow \Div_{L} X$ and $\Div^0 X \hookrightarrow \Div^0_{L} X$.
For abelian varieties $A_1,A_2$ over $K$, we will use the notation $A_1 \sim_{\mQ} A_2$ to denote that $A_1$ is $\mQ$-isogenous to $A_2$.
We will typically refer to a finite abelian group by its invariants $[n_1,\dots , n_m]$ (ordered by divisibility). 
\vspace{-.28em}
\subsection{Comments on code}
\label{ss:comments-code}
This paper has a large computational component.
We use the computer algebra programs \texttt{Maple}$^{\mathrm{TM}}$ \cite{Maple10}, \texttt{Magma} \cite{Magma}, and \texttt{Sage} \cite{sagemath} to perform these computations.
The code verifying our claims is available at the Github repository
\begin{center}
\url{https://github.com/jmorrow4692/SporadicCubicTorsion},
\end{center}
on the final author's website
\begin{center}
\url{https://math.emory.edu/~dzb/DEvHMZB-sporadicTorsion/},
\end{center}
and attached as an ancillary file on the arXiv page for this paper.
\vspace{-.28em}
\subsection{Summary of cases and techniques}
In Table \ref{table:summary}, we summarize the modular curves from Remark \ref{remark-enumeration-of-cases}, their genera, and the proof technique we use to determine the cubic points on these modular curves.

\begin{table}[h!]
\def\arraystretch{2.4}
\centering
\resizebox{\textwidth}{!}
{\begin{tabular}{|c||c||c||c|}
\hline
\textbf{\textsc{Level}} & \textbf{\textsc{Genus}} & \textbf{\textsc{Method of proof}} & $\arraycolsep=1.5pt\def\arraystretch{.5}\begin{array}{c}
\textbf{\textsc{Genus of}} \\ \textbf{\textsc{quotient}}
\end{array}$
\\
\hline
\hline
32 & 17 & Maps to another curve in this table& $g(X_1(2,16)) = 5$ \\
36 & 17 & Maps to another curve in this table & $g(X_1(2,18)) = 7$ \\
\hline
\hline
 22 & 6 & Local methods at $p=3$ (\S \ref{sec:22}) & N/A\\
25 & 12 & Local methods at $p=3$ & N/A\\
\hline
\hline
 21 & 5 & Direct analysis over $\mQ$ (\S \ref{sec:21}) & N/A\\

\hline
\hline
 26 & 10 & Direct analysis over $\mF_3$ & N/A\\
\hline
\hline
 30 & 9 & Direct analysis over $\mQ$ on $X_0(30)$ (\S \ref{sec:30}) & $g(X_0(30)) = 3$ \\
33 & 21 & Direct analysis over $\mQ$ on $X_0(33)$ & $g(X_0(33)) = 3$\\
35 & 25 & Direct analysis over $\mQ$ on $X_0(35)$ & $g(X_0(35)) = 3$\\
 39 & 33 & Direct analysis over $\mQ$ on $X_0(39)$ & $g(X_0(39)) = 3$\\

 \hline
\hline
 (2,16) & 5 & Hecke bound + direct analysis over $\mF_3$ (\S \ref{sec:216}) & N/A\\
 (2,18) & 7 & Hecke bound + direct analysis over $\mF_5$ & N/A \\
28 & 10 & Hecke bound + direct analysis over $\mF_3$ (\S \ref{subsec:X128})& N/A \\
 \hline
 \hline 
 24 & 5 & Hecke bound + additional argument (\S \ref{thm:cusp_generate}) + direct analysis over $\mF_5$ & N/A\\ 
\hline
\hline
 45 & 41 & Hecke bound + direct analysis over $\mQ$ on $X_H(45)$ ($\S \ref{sec:45}$)& $g(X_H(45)) = 5$\\
 \hline
 \hline
 65 & 121 & Formal immersion criteria ($\S$\ref{subsec:analysis65}) & $g(X_0(65)) = 5$\\
121 & 526 & Formal immersion criteria ($\S$\ref{subsec:analysis121}) & $g(X_0(121)) = 6$\\
\hline
\end{tabular}}
\vspace*{1em}
\caption{Summary of genera of modular curves and proof techniques}
\label{table:summary}
\end{table}
\clearpage

\section{Background}
\label{sec:background}

In this section, we recall some basic definitions concerning curves and the definition of certain modular curves.

\subsection{\text{Geometry of curves}}
First, we review a few definitions from the geometry of curves. Let $X$ be a smooth proper geometrically connected curve defined over a field $K$.

\begin{definition}
The \textsf{gonality} $\gamma(X)$ of $X$ is the minimal degree among all finite morphisms $X \to \P^1_K$. We say that a closed point $P \in X$ has \defi{degree $d$} if $[K(P) : K] = d$.
\end{definition}

\begin{remark}
 If $X$ admits a map $f\colon X \to \P^1_K$ of degree $d$, or a map $f\colon X \to E$ of degree $d$, where $E$ is an elliptic curve with positive rank over $K$, then $X$ also admits infinitely many points of degree $d$ (arising from fibers of $f$).
 Conversely, if $d < \gamma(X)/2$, then $X$ admits only finitely many points of degree $d$ \cite[Prop.~1]{frey:curves-with-infinitely-many}; if the Jacobian of $X$ has rank 0 over $K$ (and in particular $X$ does not admit a non-constant map to an elliptic curve with positive rank over $K$), then in fact for $d < \gamma(X)$, $X$ admits only finitely many points of degree $d$ \cite[Proposition 2.3]{derickxS:quintic-sextic-torsion}.
\end{remark}

\begin{definition}
For a positive integer $d$, we define the \cdef{$d^{\text{th}}$-symmetric power of $X$} to be $X^{(d)} := X^d/S_d$ where $S_d$ is the symmetric group on $d$ letters. The $K$-points of $X^{(d)}$ correspond to effective $K$-rational divisors on $X$ of degree $d$. In particular, a point of $X/K$ of degree $d$ gives rise to a divisor of degree $d$, and thus a point of $X^{(d)}(K)$, and we will often identify a degree $d$ point of $X$ with a divisor of degree $d$ without distinguishing notation.

If $X^{(d)}(K)$ is non-empty, then a fixed $K$-rational divisor $E$ of degree $d$ gives rise to a corresponding \textsf{Abel--Jacobi map}
\[f_{d,E}\colon X^{(d)} \to J_X, \quad D\mapsto D - E.\]
\end{definition}

\subsection{\text{Modular curves}}
We now list the various modular curves we will study.

The modular curve $X_1(M,MN)$ is the moduli space whose non-cuspidal $K$-rational points classify elliptic curves over $K$ together with independent points $P,Q$ of order $M$ and $MN$. By setting $M = 1$, we encounter the modular curves $X_1(N) := X_1(1,N)$ whose non-cuspidal $K$-rational points parametrize elliptic curves over $K$ which have a torsion point of exact order $N$ defined over $K$.

The modular curve $X_0(N)$ is the moduli space whose non-cuspidal $K$-rational points classify elliptic curves with a cyclic subgroup of order $N$ (or equivalently, a cyclic isogeny of degree $N$).

For a subgroup $\Gamma_1(N) \subseteq H \subseteq \Gamma_0(N)$, we can form the ``intermediate'' modular curve $X_H(N)$. This curve is a quotient of $X_1(N)$ by a subgroup of $\Aut (X_1(N))$, and (roughly) parameterizes elliptic curves whose mod $N$ Galois representation has image contained in $H$ (see \cite[Lemma 2.1]{RouseZB:2Adic}).

\begin{remark}[Models of $X_H$]\label{rem:modelsmodular}
For computational purposes, we need explicit equations for many of these modular curves.
For small values of $N$, the \texttt{\texttt{Magma}} intrinsic \texttt{SmallModularCurve($N$)} produces smooth models for the modular curves $X_0(N)$.
For larger values of $N$ (e.g., $N = 65, 121$), we use work of Ozman--Siksek \cite{ozman:quadraticpoints} to find canonical models for $X_0(N)$.
We use the algorithm from \cite[Section 2]{DvHZ} to compute an equation for a quotient of $X_1(45)$ via relations between modular units; see Subsection \ref{sec:45}.

There are several ways to compute models for $X_1(M, MN)$.
We use \cite{sutherland2012optimizedX1} for the equations for $X_1(N)$ and for the $j$-map $j\colon X_1(N) \rra X(1)$.
These models are generally singular.
We can compute models for $X_1(M,MN)$ by taking the normalization of (a particular component of) the fiber product $X_1(M) \times_{X(1)} X_1(MN)$. 
A priori, the fiber product produces a singular model $X$. We can desingularize using the canonical map and the \texttt{\texttt{Magma}} intrinsic \texttt{CanonicalImage(X,CanonicalMap(X))}.
We can also compute a model for $X_1(2, 2N)$ by computing a quotient of $X_1(4N)$ with \cite[Section 2]{DvHZ},
or we can use \cite[Section 3]{derickxS:quintic-sextic-torsion} (we checked that our models are birational to theirs).
\end{remark}

\subsection{\text{Cuspidal subschemes of modular curves}}
Our analysis of cubic points will heavily rely on understanding the cuspidal \textit{subscheme} of $X_1(N)$ and $X_0(N)$.

\begin{lemma}\label{lemma:cuspidal-subscheme}
Let $N \geq 5$ be a positive integer, and let $R = \mZ[1/2N]$.
\begin{enumerate}
\item The cuspidal subscheme of $X_1(N)_R$ is isomorphic to
\[
\bigsqcup_{d\mid N} (\mu_{N/d} \times \mZ/d\mZ)'/[-1],
\]
where the prime notation refers to points of maximal order.
\item The cuspidal subscheme of $X_0(N)_R$ is isomorphic to
\[
\bigsqcup_{d\mid N} (\mu_{\gcd(d, N/d)})'
\]
where the prime notation refers to points of maximal order.
\end{enumerate}
\end{lemma}

\begin{proof}
This can be directly computed from \cite[Chapter 1, Sections 1.3, 1.4, and 2.2]{derickx:thesis} and \cite[Footnote 5]{derickx2014gonality}.
\end{proof}

\subsection{\text{Modular units}}\label{subsec:modularunits}
In practice, we will work with some model of $X_1(N)$ (canonical or singular), and will need to explicitly determine
the cusps on our model (e.g., to find an explicit basis for $J_1(N)(\Q)_{\tors}$). 
The naive approach is to compute the poles of the $j$-map but it has high degree when $N$ is large. 
Modular units are a useful alternative.

\begin{definition}
A non-zero element of $\mQ(X_1(N))$ is called a \textsf{modular unit} if all of its poles and roots are cusps.
Let $\mathcal{F}_1(N) \subset \mQ(X_1(N))/\mQ^\times$ be \textsf{the group of modular units modulo $\mQ^\times$}.
\end{definition}
A basis of $\mathcal{F}_1(N)$ mod $\mQ^\times$ is given in \cite[Conjecture 1]{derickx2014gonality} which was
proved by Streng \cite[Theorem 1]{streng2015generators}.
The relevance for our purposes is that this basis is expressed in terms of the same coordinates
used in the defining equations for $X_1(N)$ from \cite{sutherland:X1NTables},
and we have their divisors as well \cite{vanHoeijHanson}.
For further discussion of modular units, we refer the reader to \cite[Section 2]{derickx2014gonality} and \cite{kubertLangModularUnits}.

\section{Modular Jacobians of rank 0}
\label{sec:computing-ranks}

In this section, we compute, with proof, the complete list of $N$ for which $J_0(N)(\Q)$ (resp., $J_1(N)(\Q)$, resp., $J_1(2,2N)(\Q)$) has rank 0. This extends the computation of \cite[Lemma 1 (3)]{derickx2014gonality} and \cite[Theorem 4.1]{derickxS:quintic-sextic-torsion}.

\begin{theorem}
	\label{Prop-J-ranks}
	Let $S_0$ be the set
	\begin{multline*}
	\{1,\ldots,36, 38, \ldots, 42, 44, \ldots, 52, 54, 55, 56, 59, 60, 62, 63, 64, 66, 68, 69, 70, 71, 72, 75, 76, 78,\\ 80,
	81, 84, 87, 90, 94, 95, 96, 98, 100, 104, 105, 108, 110, 119, 120, 126, 132, 140, 144, 150, 168, 180\},
	\end{multline*}
	and let $S_1$ be the set
	\[
	\{1, \ldots, 21, 24, 25, 26, 27, 30, 33, 35, 42, 45\}.
	\]
	Then the following are true.
	\begin{enumerate}
		\item The rank of $J_0(N)(\Q)$ is zero if $N \in S_0$.
		\item The rank of $J_1(N)(\Q)$ is zero if
		$N \in S_0 - \{63, 80, 95, 104, 105, 126, 144\}$.
		\item The rank of $J_1(2,2N)(\Q)$ is zero if $N \in S_1$. 
	\end{enumerate}
Under the assumption of the BSD conjecture the converses of parts (1), (2) and (3) are also true.
\end{theorem}

In preparation of the proof, we make a series of remarks about computing ranks of modular abelian varieties.

\begin{remark}[Analytic ranks]
	\label{remark:analytic-ranks}
	Let $A$ be a simple factor of $J_1(N)$. By Kolyvagyn's theorem (improved to $J_1(N)$ by Kato \cite[Corollary 14.3]{Kato:p-adic-hodge-zeta}) we know that if $L(A,1)$, the $L$-series of $A$ evaluated at $1$, is non-zero, then the rank of $A(\Q)$ is zero. The BSD conjecture implies that the converse of this statement is also true. A provably correct computation of whether $L(A,1) \neq 0$ is implemented in \texttt{\texttt{Magma}}: \texttt{Decomposition(JOne(N))} computes the simple factors of $J_1(N)$, and for a simple factor $A$, \texttt{IsZeroAt(LSeries(A),1)} provably computes whether $L(A,1) \neq 0$. (This is similar to the approach of \cite[Theorem 4.1]{derickxS:quintic-sextic-torsion}.)
	
	For simple factors of $J_0(N)$ this computation is quite fast, even for relatively large $N$. For factors of $J_1(N)$ this computation is much slower, even for $N$ in the $100,\ldots,180$ range.
\end{remark}

\begin{remark}[Winding quotients]
 \label{remark:winding-quotient}

A faster alternative to working directly with the $L$-series of the simple factors of $J_*(N)$, where $*$ is either 0 or 1, is the following. The \emph{winding quotient} $J_e(N)$ of $J_*(N)$ is the largest quotient with analytic rank zero, and may be described as $J_e(N) = J_*(N)/I_eJ_*(N)$, where $e = \{0,\infty\} \in H_1(X_*(N)(\mC),\Q)$ is the ``winding element'' and $I_e$ is the annihilator of $e$ in the Hecke algebra $\mathbf{T}$ (see \cite[Prop. 1]{Merel:uniformity} and \cite[Definition 4.1 and Theorem 4.2]{derickx2017small}). In particular, $J_*(N)(\Q)$ has rank zero if and only if the projection of the image of $e$ under the Hecke algebra $\mathbf{T}$ onto the cuspidal subspace spans it.

By \cite[Theorem 5.1]{LarioS:hecke}, the Hecke algebra $\mathbf{T}$ on $\Gamma_*(N)$ is generated by $T_n$ with
\[
n \leq \frac{k}{12} \cdot \left[\SL_2(\Z) \colon \Gamma_*(N)\right]
\]
(where we note that their proof, as written, also works for $\Gamma_1(N)$), so this computation is finite and easily implemented in \texttt{\texttt{Magma}}; see \texttt{master-ranks.m} for more detail.
\end{remark}

\begin{remark}[Moonshine and Ogg's Jack Daniels Challenge]
\label{remark:ogg}
 The subspace of weight 2 newforms with sign of functional equation equal to $-1$ is the subspace fixed by the Fricke involution, and thus for $p$ prime, $J_0(p)$ is isogenous to a product $A \times J_0^+(p)$, where $J_0^+(p)$ is the Jacobian of the quotient $X^+_0(p)$ of $X_0(p)$ by the Fricke involution, and where, necessarily, each factor of $A$ has even rank and each factor of $J_0(p)$ has odd rank. In particular, if $J_0(p)(\Q)$ has rank 0, then $X^+_0(p)$ has genus 0.

As is well celebrated, Ogg \cite[Remarque 1]{Ogg:Automorphismes-de-courbes-modulaires} proved that for $p$ prime, $X^+_0(p)$ has genus zero if and only if $p$ divides the order of the Monster group. Hence, if $J_0(N)(\Q)$ has rank 0 and $p$ is a prime divisor of $N$, then $p$ divides
\[
 2^{46} \cdot 3^{20} \cdot 5^9 \cdot 7^6 \cdot 11^2 \cdot 13^3 \cdot 17 \cdot 19 \cdot 23 \cdot 29 \cdot 31 \cdot 41 \cdot 47 \cdot 59 \cdot 71
\]
(and offered a bottle of Jack Daniels for an explanation of the coincidence).

 We note that this approach does not generalize to composite level $N$. For example, $J_0(28)(\Q)$ has rank 0, but $X_0^+(28)$ is the elliptic curve $X_0(14)$. It is still true that one can deduce the signs of the functional equations of newforms via Atkin--Lehner, but the presence of oldforms contributes additional genus to $X_0^+(N)$.

\end{remark}

\begin{proof}[Proof of Theorem \ref{Prop-J-ranks}]

For $N \in S_0$, we compute that $J_0(N)(\Q)$ has rank zero via Remark \ref{remark:analytic-ranks} (using Remark \ref{remark:winding-quotient} to check our computation), and for any integer of the form $N\cdot p$ with $N \in S_0$ and $p$ a prime divisor of the order of the Monster group, we again similarly compute that the rank is non-zero via \texttt{\texttt{Magma}}. This proves part (1).

For part (2), if $J_1(N)(\Q)$ has rank 0, then $J_0(N)(\Q)$ also has rank 0 (but not conversely), so we again check using Remark \ref{remark:analytic-ranks} (and Remark \ref{remark:winding-quotient} for some of the larger $N$) for which $N \in S_0$ $J_1(N)(\Q)$ has rank 0.

For part (3), by \cite[Section 1]{JeonK:bielliptic-modular-curves}, there is an isomorphism $X_{\Delta}(4N) \to X_1(2,2N)$, with $\Delta := \{\pm 1, \pm (2N+1)\}$, and thus a surjection $X_1(4N) \to X_1(2,2N)$.
Thus, if $J_1(4N)(\Q)$ has rank 0, then $J_1(2,2N) (\Q)$ also has rank 0. On the other hand, there are surjective maps $X_1(2,2N) \cong X_{\Delta}(4N) \to X_0(4N)$ and $X_1(2,2N) \to X_1(2N)$, so if either $J_0(4N) (\Q)$ or $J_1(2N) (\Q)$ have positive rank, then $J_1(2,2N) (\Q)$ also has positive rank. Cases (1) and (2) thus determine the rank of $J_1(2,2N) (\Q)$ unless $N = 20,26,36$. The remaining three cases we do by hand using the isomorphism $X_{\Delta}(4N) \to X_1(2,2N)$ and \texttt{Magma}'s intrinsic \verb+JH+ to compute the rank of $J_{\Delta}(4N)$ via Remark \ref{remark:analytic-ranks}.
 \end{proof}
See the file \texttt{master-ranks.m} for code verifying these computations.

\section{Computing rational torsion on modular Jacobians}
\label{sec:computing-torsion}
This section explains how to compute the rational torsion of a modular Jacobian. For an additional technique see \cite[Section 4]{ozman:quadraticpoints}.

\subsection{Local bounds on torsion via reduction}
\label{ssec:local-torsion-bounds}

Let $K$ be a number field, $\mathfrak{p}$ a prime of $K$ lying over a rational prime $p > 2$, and $A/K$ an abelian variety with good reduction at $\mathfrak{p}$. Then by \cite[Appendix]{katz:galois-properties-of-torsion}, if the ramification index $e_{\mathfrak{p}}$ is less than $p-1$, the reduction map $A(K)_{\tors} \rra A(\mF_{\mathfrak{p}})$ is injective. The GCD of $\#A(\mF_{\mathfrak{p}})$, as $\mathfrak{p}$ ranges over such primes, gives a naive upper bound on $A(K)_{\tors}$. For rational torsion on modular Jacobians, one can quickly compute $\#A(\mF_{{p}})$ via the coefficients of the corresponding modular forms, which \texttt{Magma} has packaged into the intrinsic \texttt{TorsionMultiple}.

Comparing orders is usually insufficient (e.g., if $A(\Q)_{\tors}$ is cyclic and $A(\mF_{{p}})$ is non cyclic, the output of \texttt{TorsionMultiple} is usually larger than $\#A(\Q)_{\tors}$; see Example \ref{example:21-torsion}). To improve these bounds, we compute the ``GCD'' of the groups $A(\mF_{{p}})$ for various ${p}$; more precisely, given abelian groups $A_1,\ldots, A_n$, we define the $\GCD(A_1,\ldots,A_n)$ to be the largest abelian group $A$ such that each $A_i$ contains a subgroup isomorphic to $A$.

\begin{example}[Torsion on $J_1(21)$]
 \label{example:21-torsion}
 To demonstrate this, we consider the rational torsion on $A = J_1(21)$, which is a $5$-dimensional modular abelian variety. The output of the intrinsic \texttt{TorsionMultiple(JOne(21))} yields a bound of $728$, but $A(\mF_{5})$ has invariants $[ 2184 ]$ and $A(\mF_{11})$ has invariants $[14, 6916]$, and these groups have GCD $[364]$. We computed $A(\mF_{p})$ by reducing a model for $X_1(21)$ modulo $p$ and using \texttt{\texttt{Magma}} 's \texttt{ClassGroup} intrinsic; see the \texttt{\texttt{Magma}} file \texttt{\texttt{master-21.m}}.
\end{example}

\subsection{Better local bounds on torsion via Eichler--Shimura and modular symbols}
\label{ssec:better-local-torsion-bounds}

The naive approach above does not incorporate the action of complex conjugation, and slows quickly as the genus grows. We describe here a substantial improvement, derived from the Eichler--Shimura relation.
For an intermediate modular curve $X_H(N)$, let $J_H(N)$ denote its Jacobian. 

For a prime $q \nmid 2N$, let $T_q$ be the $q$-th Hecke operator. 
By the Eichler--Shimura relation, the kernel of the operator
\[
 T_q-q\diamondop{q}-1 \colon J_H(N)(\Qbar)_{\tors} \to J_H(N)(\Qbar)_{\tors}
\]
contains the prime to $q$ torsion in $J_H(N)(\Q)$ (cf.~\cite[Proposition 5.2]{derickx2017small} and \cite[page 87]{DiamondI:modularCurves}).
Additionally, let
\[
 \tau\colon J_H(N)(\overline{\Q}) \to J_H(N)(\overline{\Q})
\]
be complex conjugation. Then it is clear that $\tau-1$ vanishes on $J_H(N)(\Q)$, and thus also vanishes on $J_H(N)(\Q)_{\tors}$.

Thus, for a finite set of primes $q_1,\ldots,q_n$ (still coprime to $2N$) with $n \geq 2$, 
\begin{equation}\label{eqn:MH}
M_H := J_H(N)(\overline{\Q})_{\tors}[T_{q_1}-q_1\diamondop{q_1}-1,\ldots,T_{q_n}-q_n\diamondop{q_n}-1,\tau-1]
\end{equation}
contains $J_H(N)(\Q)_{\tors}$.
We can efficiently compute $M_H$ as follows. Under the uniformization
\[
J_H(N)(\C) \cong H_1(X_H(N)(\C),\C)/H_1(X_H(N)(\C),\Z)
\]
we can identify the geometric torsion as
\[
J_H(N)(\overline{\Q})_{\tors} \cong H_1(X_H(N)(\C),\Q)/H_1(X_H(N)(\C),\Z).
\]
While it does not make sense to ask this identification to be Galois equivariant, it does commute with the Hecke and diamond operators, and with complex conjugation. The right hand side lends itself very well to explicit computations with \texttt{Sage} using modular symbols (and is very fast, since the computations are now essentially linear algebra). This is implemented in the module \texttt{CuspidalClassgroup} from \cite{derickx:X1Ncode}, and is also available in our accompanying code (see \ref{ss:comments-code}).

We will call this upper bound on rational torsion we get from computing $M_H$ the \cdef{Hecke bound}. \\

To demonstrate how the Hecke bound produces better bounds than the local bounds from Subsection \ref{ssec:local-torsion-bounds}, we will consider the following two examples. 

\begin{example}[Torsion on $J_1(28)$]\label{exam:Hecke28}
Consider the rational torsion on $A = J_1(28)$, which is a $10$-dimensional modular abelian variety. The intrinsic \texttt{TorsionMultiple(JOne(28))} produces a bound of $359424$. We compute that $A(\mF_{3})$ has invariants $[ 4, 4, 24, 936]$ and $A(\mF_{5})$ has invariants $[ 2, 2, 8, 8, 312, 936]$, which have GCD $[ 4,4,24,936]$. 
The Hecke bound give an upper bound of $[2, 4, 12, 936]$, which improves upon the GCD bound; see the \texttt{Sage} file \texttt{torsionComputations.py}. 
\end{example}

\begin{example}[Torsion on $J_1(2,18)$]\label{exam:Hecke218}
Consider the rational torsion on $A = J_1(2,18)$, which is a $7$-dimensional modular abelian variety. We compute that $A(\mF_{5})$ has invariants $[ 6, 84, 252 ]$ and $A(\mF_{7})$ has invariants $[ 3, 6, 6, 126, 126]$, which have GCD $[6,42,126]$. 
The Hecke bound give an upper bound of $[2, 42,126]$, which improves upon the GCD bound; see the \texttt{Sage} file \texttt{torsionComputations.py}. 
\end{example}

\subsection{Cuspidal torsion and a generalized Conrad--Edixhoven--Stein conjecture}
\label{ssec:cuspidal-torsion}
By the Manin--Drinfeld theorem \cite{maninParabolic, drinfeldTwotheorems}, cuspidal divisors are torsion; conversely, it is expected that $J(\Q)_{\tors}$ is cuspidal for a modular Jacobian $J$.
More precisely, we conjecture that the torsion on $J_1(N)(\Q)$ is generated by the $\Gal(\overline{\Q}/\Q)$-orbits of cusps (see Conjecture \ref{conj:torsion}).

This is a conjecture of Conrad--Edixhoven--Stein \cite[Conjecture 6.2.2]{ConradES:J-connected-fibers} for the modular Jacobian $J_1(p)$ where $p$ is a prime. The authors prove this conjecture for all primes $p \leq 157$ except for $p = 29,97,101,109,$ and $113$, and the case of $ p =29$ was proved in \cite[Theorem 6.4]{derickx2017small}. Moreover, their conjecture is true for all primes $p$ such that $J_1(p)(\Q)$ has rank 0.
For composite $N$, the torsion subgroup of $J_1(N)$ generated by $\Gal(\overline{\Q}/\Q)$-orbits of cusps has been studied, and in some special cases, the Conrad--Edixhoven--Stein conjecture has been proved; see the following works 
\cite{yu:cuspidalclassnumber, takagi:classnumberp, takai:cuspidalclassnumber3power, hazama:cuspidalclassnumber,  csirik:KernelEisensteinIdeal, takagi:classnumber2power, yang:modularunits, yang:CuspidalTorsion, sun:cuspidalclassnumber, chen:CuspidalTorsion, takagi:CupsidalClassNumber2p, ohta:CuspidalTorsion, takagi:CuspidalTorsion2p}.

In this subsection, we will use the construction of $M_H$ from Equation \eqref{eqn:MH} and the results from Subsection \ref{subsec:modularunits} to determine when the rational torsion on certain modular Jacobians is cuspidal. 
First, we need to establish some definitions.

\begin{definition}\label{def:generatedcusps}
We define the following subgroups of $\Div X_H(N)$.
\begin{itemize}
\item Let $\Cusp X_H(N)$ denote the free abelian group generated by the cusps of $X_{H_{\overline{\mathbb{Q}}}}$.
\item Let $\Cusp_\Q X_H(N) := \Cusp X_H(N) \cap \, \Div_{\Q} X_H(N)$ denote the subgroup generated by the $\Gal(\overline{\Q}/\Q)$-orbits of cusps.
\end{itemize}
\end{definition}

\begin{definition}
We define the following quotients of the divisor groups from Definition \ref{def:generatedcusps}.
\begin{itemize}
\item Let $\,\PrinCusp X_H(N)$ denote the principal divisors in $\Cuspo X_H(N)$, and let $\ClCusp X_H(N):= \Cuspo X_H(N) /\PrinCusp X_H(N)$ denote the quotient.
\item Let $\PrinCusp_\Q X_H(N)$ denote the principal divisors in $\Cuspo_\Q X_H(N)$, and let $\ClCusp_\Q X_H(N):= \Cuspo_\Q X_H(N) /\PrinCusp X_H(N)$ denote the quotient.
\end{itemize}
 \end{definition}

With these definitions it is clear that determining if the rational torsion on the modular Jacobian $J_H$ is cuspidal boils down to whether $\ClCusp_\Q X_H(N)= J_H(\Q)_{\tors}$.

Subsection \ref{ssec:better-local-torsion-bounds} gave an inclusion $J_H(\mQ)_{\tors} \subset M_H$ (the ``Hecke bound'').
In order to get lower bounds, we need to compute $\ClCusp_\Q X_H(N)$, which we now describe. 
Let $G:=\Gal(\Q(\zeta_N)/\Q) \cong (\Z/N\Z)^\times$ be the Galois group of the cyclotomic field obtained by adjoining an $N$-th root of unity to $\Q$. 
To determine the group $\ClCusp_\Q X_H(N)$, we first use results and code from \cite{derickx2014gonality} and actions of diamond operators to compute $\ClCusp X_H(N)$, and from here we take $G$-invariants as $\Cuspo_{\Q} X_H(N)=(\Cuspo X_H(N))^G$. 
An implementation to compute $\ClCusp_\Q X_1(N)$ is available at \cite{vanHoeij:X1Ncode}, explained in \cite{vanHoeijHanson},
and a \texttt{Sage} version extending this to $\ClCusp_\Q X_H(N)$ can be found in the module \texttt{MiscellaneousFunctions} from \cite{derickx:X1Ncode}.
We illustrate the utility of this code in the following examples.

\begin{example}[Torsion on $J_1(28)$]
The code finds $\ClCusp_\Q X_1(28) \cong [2, 4, 12, 936]$. 
Combining this with the Hecke bound from Example \ref{exam:Hecke28} gives $\ClCusp_\Q X_1(28) = J_1(28)(\Q)_{\tors}$, and in particular, all of the rational torsion on $J_1(28)$ is cuspidal; see the \texttt{Sage} file \texttt{torsionComputations.py}. 
\end{example}

\begin{example}[Torsion on $J_1(2,18)$]
The code finds $\ClCusp_\Q X_1(2,18) \cong [2, 42, 126]$. 
Combining this with the Hecke bound from Example \ref{exam:Hecke218} gives $\ClCusp_\Q X_1(2,18) = J_1(2,18)(\Q)_{\tors}$, and in particular, all of the rational torsion on $J_1(2,18)$ is cuspidal; see the \texttt{Sage} file \texttt{torsionComputations.py}. 
\end{example}

To study the question of whether the rational torsion on the modular Jacobian is cuspidal in more generality, we proceed as follows. 
Since $(\ClCusp X_H(N))^{G} \subset J_H(\Q)$, it is necessary that the map 
\[
\Cuspo_{\Q} X_H(N)=(\Cuspo X_H(N))^G \to (\ClCusp X_H(N))^G
\]
is surjective in order for the rational torsion to be cuspidal. 

For the remainder of this section, we will focus our attention on determining the rational torsion on modular curves $X_1(N)$ and $X_1(2,2N)$ with modular Jacobians $J_1(N)$ and $J_1(2,2N)$, respectively. 
As a first step, we show that for $N \leq 55$, the above map on divisors is surjective. 

\begin{proposition}\label{prop:surjectivity}
Let $N \leq 55$ be an integer. Then $\ClCusp_\Q X_1(N)=(\ClCusp X_1(N))^G$.
\end{proposition}
\begin{proof}
We verify the result  by computing the maps 
\[(\Cuspo X_1(N))^G \to (\ClCusp X_1(N))^G\]
in terms of modular symbols and the computation showed it was surjective in all cases. 
\end{proof}

For code verifying these claims as well as those in Theorem~\ref{thm:cusp_generate},  Corollary~\ref{coro:almostallN}, and Proposition~\ref{prop:cusp_generate_2N}
see the \texttt{Sage} file \texttt{torsionComputations.py}. \\

For the main theorem in this section, we will also need the following lemma.

\begin{lemma}\label{lem:equality_test}
Let $R$ be a commutative ring $M$ be an $R$-module whose cardinality is finite and $C, M', T$ be $R$-submodules of $M$ such that $C\subseteq M'$, $C \subseteq T$ and $M' \cap T = C$. Assume that for all maximal ideals $m$ of $R$ one has $(M/C)[m] = (M'/C)[m]$, then $T=C$.
\end{lemma}
\begin{proof}
By taking the quotient by $C$, one may assume that $C=0$.
Since $T$ is of finite cardinality, it is isomorphic to $\otimes_{i=1}^k T[m_i^\infty]$ for some finite sequence of maximal ideals $m_1,\dots ,m_k$. Let $m=m_i$ be one of these maximal ideals. Due to the equalities $M[m]=M'[m]$ and $M'\cap T =0$, one gets $T[m] = M[m]\cap T = M'[m]\cap T =0$, and thence $T[m^\infty] =0$. Since this holds for all $m_i$, $T=0$.
\end{proof}

\begin{theorem}\label{thm:cusp_generate}
Let $N \leq 55$, $N \neq 54$ be an integer. Then
\[
\ClCusp_\Q X_1(N)= J_1(N)(\Q)_{\tors}
\]
If $N = 54$, then the index of $\ClCusp_\Q X_1(N)$ in $J_1(N)(\Q)_{\tors}$ is a divisor of $3$.
\end{theorem}

\begin{proof}
This was verified using a \texttt{Sage} computation, which we now describe. 
We compute the $M$ from Equation \eqref{eqn:MH} for the modular Jacobians $J_1(N)$ for $N \leq 55$. 
For $N \neq 24,32,33,40,48,54$, our computation shows that $M \subseteq \ClCusp X_1(N)$ holds and hence
\[
(\ClCusp X_1(N))^G \subseteq J_1(N)(\Q)_{\tors} \subseteq M^G \subseteq (\ClCusp X_1(N))^G. 
\]
Therefore, the result follows from Proposition \ref{prop:surjectivity} except for the cases of $N = 24, 24,32,33,40,48,54$. 
For these cases, define $M' =M \cap \ClCusp X_1(N)$. The \texttt{Sage} computation then shows that for all primes $p$ that divide $\#M$, we have $(M/\ClCusp_\Q X_1(N))[p] = (M'/\ClCusp_\Q X_1(N))[p]$ holds for all the above $N$ except for $54$. The theorem follows by applying Lemma \ref{lem:equality_test} with $C=\ClCusp_\Q X_1(N)$ and $T=J_1(N)(\Q)_{\tors}$.

For the $N = 54$, we computed the index $i:=[M+\ClCusp X_1(N):\ClCusp X_1(N)]$; taking $\Gal(\overline{\mQ}/\mQ)$-invariants of the sequence 
\[0 \to \ClCusp X_1(N) \to M+\ClCusp X_1(N) \to (M+\ClCusp X_1(N))/\ClCusp X_1(N) \to 0\] 
shows that $[(M+\ClCusp X_1(N))^{\Gal(\overline{\mQ}/\mQ)}:(\ClCusp X_1(N))^{\Gal(\overline{\mQ}/\mQ)}]$ divides $3$. 
Since $J_1(\Q)_{\tors} =(M+\ClCusp X_1(N))^{\Gal(\overline{\mQ}/\mQ) }$ and $\ClCusp X_1(N)^{\Gal(\overline{\mQ}/\mQ)}=\ClCusp_\Q X_1(N))$, we use this index to get the multiplicative upper bound of $3$ mentioned in the second statement.
\end{proof}

The only $N \leq 55$ such that $J_1(N)$ has positive rank are $37,43$ and $53$ so as a corollary we immediately get that:

\begin{corollary}\label{coro:almostallN}
Let $N \leq 55$ be an integer. 
If $N \neq 37,43,53,$ or $54$, then 
\[
\ClCusp_\Q X_1(N)= J_1(N)(\Q).
\] 
\end{corollary}

Based on the results of Theorem \ref{thm:cusp_generate}, we make the following conjecture, which is a generalization of the conjecture of Conrad--Edixhoven--Stein \cite[Conjecture 6.2.2]{ConradES:J-connected-fibers}. 

\begin{conjecture}\label{conj:torsion}
For any integer $N$, we have that 
\[
\ClCusp_\Q X_1(N)= J_1(N)(\Q)_{\tors}. 
\]
\end{conjecture}

Using the same proof as in Theorem \ref{thm:cusp_generate}, we can prove the following. 

\begin{proposition}\label{prop:cusp_generate_2N}
Let $N \leq 16$ be an integer. Then
\[\ClCusp_\Q X_1(2,2N)= J_1(2,2N)(\Q)_{\tors}.\]
\end{proposition}

Again see the \texttt{Sage} file \texttt{torsionComputations.py}.


\subsection{Local to global failures}
\label{ssec:local-to-global-failures}
For $J_0(N)$, the local methods described above often give only an upper bound instead of the true torsion.
With enough work one can sometimes fix this; see Example~\ref{exam:torsionJ030} below. For another compelling example
see \cite[Subsection 5.5]{ozman:quadraticpoints}  --- the curve $X_0(45)$ is a plane quartic, and they use the explicit description of $J_0(45)(\Qbar)[2]$ via bitangents to bridge the discrepancy between their local bounds and the true torsion.

\begin{example}[Torsion on $J_0(30)$]\label{exam:torsionJ030}
 Consider the rational torsion on $A = J_0(30)$, which is a $3$-dimensional modular abelian variety. The rational cuspidal divisors generate a subgroup with invariants $[2, 4, 24]$. Locally, $A(\mF_{7})$ has invariants $[2, 2, 4, 48]$ and $A(\mF_{23})$ has invariants $[2, 12, 24, 24]$; these groups have GCD $[2, 2, 4, 24]$. 
 The Hecke bound and the additional argument from Theorem \ref{thm:cusp_generate} do not improve this.

 The curve $X_0(30)$ is hyperelliptic and admits the model $y^2 = f(x)$, where
 \[ f(x) = 
 (x^2 + 3x + 1)\cdot
 (x^2 + 6x + 4)\cdot
 (x^4 + 5x^3 + 11x^2 + 10x + 4),
 \]
 and we can exploit the explicit description of the 2-torsion of the Jacobian of a hyperelliptic curve.
 The geometric 2-torsion $J_0(30)(\Qbar)[2]$ is Galois-equivariantly in bijection with even order subsets of the set of Weierstrass points modulo the subset of all Weierstrass points (see e.g., \cite[Section 6]{gross:Hanoi-lectures-on-the-arithmetic-of-hyperelliptic-curves}), and one can compute the rational torsion by taking Galois invariants. See the file \texttt{functions.m} for a short routine \texttt{twoTorsionRank} which computes $J(\Q)[2]$ for the Jacobian of a hyperelliptic curve. Using this routine, we find that $\rank_{\mF_2} J_0(30)(\Q)[2] = 3$, and hence the rational torsion on $J_0(30)$ is cuspidal and isomorphic to $[2,4,24]$. 

Note that if $J$ is the Jacobian of a hyperelliptic curve defined by an odd degree polynomial $f(x)$, $\rank_{\F_2}J(\Q)[2] = i$, where $i+1$ is the number of factors $f(x)$; this is no longer true in general, as $J_0(30)$ demonstrates. \\

 See the \texttt{Magma} file \texttt{master-30.m} for more details.
\end{example}

\subsection{Tables of cuspidal torsion subgroups for $J_1(N)$ and $J_1(2,2N)$}
\label{ssec:local-to-global-failures}
We conclude this section by giving the structure of $\ClCusp_\Q X_1(N)$ for $10 \leq N \leq 55$ and of $\ClCusp_\Q X_1(2,2N)$ for $5\leq N \leq 16$ in terms of its invariant factor decomposition in Tables \ref{table:cuspidaltorsion} and \ref{table:cuspidaltorsion22N}. 
For code, see the \texttt{Sage} file \texttt{torsionComputations.py}. 

\begin{remark}
In \cite{yang:modularunits}, Yang determines the cuspidal torsion on $J_1(N)$ generated by $\infty$-cusps (cf.~\cite[Definition 5.4]{derickx2017small}). 
By comparing his results with Table \ref{table:cuspidaltorsion}, we can find $N$ for which one needs non-$\infty$ cusps and/or non-rational cusps to generate the torsion on $J_1(N)(\Q)$. 
For example, by \cite[Table 1]{yang:modularunits}, the $\infty$-cusps on $X_1(20)$ generate a subgroup isomorphic to $[20]$, but by Theorem \ref{thm:cusp_generate} and Table \ref{table:cuspidaltorsion}, we see that $J_1(20)(\Q)$ is cuspidal and isomorphic to $[60]$; in particular, the $\infty$-cusps do not generate all of the rational torsion on $J_1(20)$. 
A similar situation occurs for $X_1(21)$, which we discuss in Subsection \ref{sec:21}. 
\end{remark}

\begin{center}
\begin{table}[h!]
\def\arraystretch{1.3}
\centering
\resizebox{\textwidth}{!}
{\begin{tabular}{|c|c||| c|c ||| c| c |}
\hline
$N$ & $ \ClCusp_\Q X_1(N)$ & $N$ & $ \ClCusp_\Q X_1(N)$ & $N$ & $ \ClCusp_\Q X_1(N)$ \\
\hline \hline
	11 & [5]			& 27 & [3,3,52497]	& 42 & [182,1092,131040] \\					
	13 & [19]		&	28 & [2,4,12,936]&	43 & [2,1563552532984879906]\\					
	14 & [6]			&	29 & [4,4,64427244] &	44 & [4,620,3100,6575100]\\					
	15 & [4]			&30 & [4,8160]	&45 & [3,9,36,16592750496]\\					
	16 & [2,10]		& 	31 & [10,1772833370]	&46 & [408991,546949390174] \\					
	17 & [584]		&	32 & [2,2,2,4,120,11640]&	47 & [3279937688802933030787]\\					
	18 & [21]		&	33 & [5,42373650]	& 48 & [2,2,2,2,4,40,40,240,1436640]\\					
	19 & [4383]	&34 & [8760,595680]	&49 & [7,52367710906884085342]\\					
	20 & [60]		&35 & [13,109148520]&	50 & [5,1137775,47721696825]\\					
	21 & [364]		&	36 & [12,252,7812]	&51 & [8,1168,7211322610146240]\\					
	22 & [5,775]	&	37 & [160516686697605]	 & 52 & [4,28,532,7980,17470957140]\\					
	23 & [408991]&	38 & [9,4383,33595695]	&53 & [182427302879183759829891277]\\					
	24 & [2,2,120] &39 & [7,31122,3236688]	&\mbox{$54'$\hspace{-2.7pt}}& [3,3,3,9,9,1102437,1529080119]\\			
	25 & [227555]&40 & [2,2,2,8,120,895440]	&55 & [5,550,8972396739917886000]\\					
	26 & [133,1995]	& 41 & [107768799408099440] & {} & {} \\	
\hline				
\end{tabular}}
\vspace{3pt}
\caption{Cuspidal torsion in $J_1(N)(\Q)$ for $N = 11$ and $ 13 \leq N\leq 55$ (all torsion if $N \neq 54$).}
\label{table:cuspidaltorsion}
\end{table}
\end{center}


\begin{center}
\begin{table}[h!]
\def\arraystretch{1.1}
\centering
\resizebox{\textwidth}{!}
{\begin{tabular}{|c|c||| c|c ||| c| c |}
\hline
$(2,2N)$ & $ \ClCusp_\Q X_1(2,2N)$ & $(2,2N)$ & $ \ClCusp_\Q X_1(2,2N)$ & $(2,2N)$ & $ \ClCusp_\Q X_1(2,2N)$ \\
\hline \hline
$(2,10)$ &  [6]		& $(2,18)$ & [2,42,126]		& $(2,26)$	& 	 [2,14,266,3990,11970]	\\
$(2,12)$ &  [4]		& $(2,20)$ & [4,60,120]		& $(2,28)$	& 	 [2,4,4,4,8,24,936,936]	\\
$(2,14)$ &  [2,2,6,18]& $(2,22)$ & [2,10,1550,4650]	& $(2,30)$	&	 [2,2,2,2,2,24,8160,8160]	\\
$(2,16)$ & [2,20,20]	& $(2,24)$ & [2,4,4,120,240]	& $(2,32)$ &	 [2,2,2,4,4,24,120,23280,23280]	\\
\hline
\end{tabular}}
\vspace{3pt}
\caption{(Cuspidal) torsion on $J_1(2,2N)(\Q)$ for $5 \leq N\leq 16$.}
\label{table:cuspidaltorsion22N}
\end{table}
\end{center}

\section{Methods for determining cubic points on curves}
\label{sec:methods}
In this section, we describe the variety of methods we utilize to determine the cubic points on the modular curves $X_1(N)$.

\subsection{Local methods}\label{subsec:local}
In some cases, we can deduce that there are no non-cuspidal cubic points on $X_1(N)$ simply by determining points on $X_1(N)(\mF_{p^i})$ for $i = 1,2,$ and $3$ for some prime $p\nmid 2N$.

Suppose that $X$ is a curve of gonality at least 4 and at least one rational point, and that its Jacobian $J_X$ satisfies rank $J(\Q)$ = 0. Fix a point $\infty \in X(\Q)$. Since $X$ has gonality at least 4, the Abel--Jacobi maps
\[
f_{i,i\infty}\colon X^{(i)} \hookrightarrow J_X, \,D \mapsto D - i\infty
\]
are injective for $i = 1,2,3$. Moreover, since the rank of $J_X(\mathbb{Q})$ is zero, the reduction map $J_X(\mathbb{Q}) \hookrightarrow J_X(\mathbb{F}_p)$ is injective for $p > 2$ (see Remark \ref{ssec:local-torsion-bounds}). We thus get a commutative diagram
\[
\xymatrix{
X^{(i)}(\Q) \ar@{^{(}->}[r] \ar@{^{(}->}[d] & J_X (\Q) \ar@{^{(}->}[d] \\
X^{(i)}(\F_p) \ar@{^{(}->}[r] & J_X (\F_p)
}
\]
of injections. In particular, the reduction maps $X^{(i)}(\Q) \hookrightarrow X ^{(i)} (\F_p)$ are injective for $i = 1,2,3$; if there is a prime $p$ such that these maps are also surjective, then we have determined $X^{(i)}(\Q)$. One can verify surjectivity by checking cardinalities, i.e., checking if
\[
\#X^{(i)}(\Q) \geq \#X ^{(i)} (\F_p).
\]
In practice, $\#X ^{(i)} (\F_p)$ can be determined very quickly in \texttt{Magma} (even using a singular model, by working with places instead of points), and we often a priori have a lower bound on $\#X^{(i)}(\Q)$ coming from an explicit description of cusps (e.g., Lemma \ref{lemma:cuspidal-subscheme}).
 \\

See Subsection \ref{sec:22} for examples.

\subsection{Direct analysis of preimages of an Abel--Jacobi map}
\label{subsec:direct}
When $X_1(N)$ has gonality at least 4 and $J_1(N)(\mQ)$ has rank 0, one can compute the finitely many preimages of an Abel--Jacobi map $\iota\colon X_1(N)^{(3)}(\Q) \to J_1(N)(\Q)$. For values of $N$ such that the genus of $X_1(N)$ and the size of $J_1(N)(\Q)$ is not too large, it is possible to do this directly over $\Q$: fixing a base point $\infty \in X_1(N)(\Q)$, a divisor $D \in J_1(N)(\Q)$ is in the image of the Abel--Jacobi map $E \mapsto E - 3\infty$ if and only if the linear system $|D + 3 \infty| \neq \emptyset$. One can compute $|D + 3 \infty|$ via \texttt{Magma}'s \verb+RiemannRoch+ intrinsic. If $|D + 3 \infty| = \emptyset$ then we disregard it; otherwise it will contain a single effective divisor $E$ of degree 3. Thus as $D$ ranges over $J_1(N)(\Q)$, we eventually compute all of the effective degree 3 divisors (and hence the image of Abel--Jacobi). We will refer to this as \textsf{direct analysis over $\Q$}.

Direct analysis over $\Q$ can be very slow (cf.~Remark \ref{rem:timing}). A much faster method from \cite[Footnote~1]{vanHoeij:LowDegreePlaces}
works as follows. The diagram from Subsection \ref{subsec:local}
\[
 \xymatrix{
 X ^{(3)} (\Q) \ar@{^{(}->}[r]^{\iota}\ar@{^{(}->}[d]^{\red_X}& \ar@{^{(}->}[d]^{\red_J} J_X(\Q) \\
 X ^{(3)} (\F_p) \ar@{^{(}->}[r]^{\iota_p}& J_X(\F_p)
 }
 \]
 commutes, so the image of $\iota_p$ contains the reduction of the image of $\iota$. It thus suffices to:
 \begin{enumerate}
 \item compute the image of $\iota_p$ (which is generally very fast),
 \item compute the preimage of $\im \iota_p$ under $\red_J$ (also fast), and
 \item compute which elements of $\red_J^{-1}\left(\im \iota_p\right)$ are in the image of $\iota$.
 \end{enumerate}
We will refer to this approach as \textsf{direct analysis over $\F_p$}.
The set $\red_J^{-1}\left(\im \iota_p\right)$ of divisors which are ``locally in the image of Abel--Jacobi'' is generally much smaller than $J_X(\Q)$, so step (3) is much faster. (Equivalently, since each map is injective, one can compute the intersection $\left(\im \red_J\right) \cap \left(\im \iota_p\right).$)
One could further speed up the computation by repeating this procedure at several primes; see the Mordell--Weil sieve \cite[Section 6]{ozman:quadraticpoints}. Implementing this was unnecessary for our results. \\

See Subsections \ref{sec:21} and \ref{sec:216} for examples.

\subsection{Direct analysis on non-trigonal curve quotients}
\label{subsec:directanalysisquotient}
While direct analysis works in principle, we encounter many values of $N$ where the genus of $X_1(N)$ and the size of $J_1(N)$ are large (for example, $g(X_1(45)) = 41$ and $16128153482112 \mid \#J_1(45)(\Q)$), and hence working directly with $X_1(N)^{(3)}(\Q)$ and $J_1(N)(\Q)$ is hard. 
Instead, we consider a morphism $X_1(N)\to X$ where $X$ is non-trigonal and perform the direct analysis on $X$. For details on how to find a model for $X$, see Remark \ref{rem:modelsmodular}.

\begin{remark}[Cubic points on hyperelliptic curves]\label{rem:cuspidalgonality}
A curve $X$ of genus $g\leq 2$ is trigonal as it admits a base-point free $g^1_3$ by Riemann--Roch, and a non-hyperelliptic curve of genus $g = 3$ or $ 4$ is trigonal. However, a hyperelliptic curve of $g\geq 3$ is not trigonal, and thus has finitely many cubic points. In this case, $X^{(3)}(\Q)$ is still infinite; it contains the image of $X(\Q) \times X^{(2)}(\Q)$ (i.e., divisors of the form $P + Q + Q^{\iota}$, where $P \in X(\Q)$ and where $\iota$ is the hyperelliptic involution), and the complement of this subset contains the finitely many cubic points.
Therefore, when we search for non-trigonal curve quotients, we either need a genus $g\geq 3$ hyperelliptic curve or certain non-hyperelliptic curves of $g\geq 5$.

In the hyperelliptic case, we need to know that the image of a cubic point does not reduce to a non-cuspidal rational point plus another divisor.
We encounter four genus 3 hyperelliptic cases, namely $X_0(N)$ for $N = 30,33,35,$ and $39$.
By \cite[Theorem 1]{kenkuIsomClasses}, we know that the rational points on the curves $X_0(N)$ for these $N$ are all cuspidal.
\end{remark}

See Subsection \ref{sec:30} for examples.

\subsection{Formal immersion criteria}
\label{ss:form-immers-crit}
When the rank of $J_1(N)(\mQ)$ is positive, we apply formal immersion criteria of \cite{derickx2017small} to determine all of the points on $X_1(N)^{(3)}(\mQ)$.
The underlying ideas of using formal immersion in the study of rational and higher degree points on modular curves comes from the foundational works of Mazur \cite{mazur1978primedegreeisogeny} and Kamienny \cite{kamienny:torsionpointshigherdegree}. 

To begin, we define formal immersions. Throughout this subsection, let $R$ be a discrete valuation ring with perfect residue field $\kappa$ and fraction field $K$.

\begin{definition}
Let $\phi\colon X \to Y$ be a morphism of Noetherian schemes and $x\in X$ a point and $y = f(x) \in Y$.
Then $\phi$ is a \textsf{formal immersion at $x$} if the induced morphism of complete local rings $\widehat{\phi}^*\colon \widehat{\mathcal{O}_{Y,y}} \to \widehat{\mathcal{O}_{X,x}}$ is surjective.
\end{definition}

The main reason we consider formal immersions is the following lemma.

\begin{lemma}[\protect{\cite[Lemma 2.2]{derickx2017small}}]\label{lem:formalimmersion}
Let $X,Y$ be Noetherian schemes. Let $R$ be a Noetherian local ring with maximal ideal $\mathfrak{m}$ and residue field $\kappa = R/\mathfrak{m}$.
Suppose that $f\colon X\to Y$ is a formal immersion at a point $x\in X(\kappa)$ and suppose that $P,Q \in X(R)$ are two points such that $x = P_{\kappa} = Q_{\kappa}$ and $f(P) = f(Q)$. Then $P=Q$.
\end{lemma}

To verify that a morphism is a formal immersion at a point, we will use the following criterion.

\begin{proposition}[\protect{\cite[Proposition 3.7]{derickx2017small}}]\label{prop:generalKamienny}
Let $C$ be a smooth projective curve over $R$ with geometrically connected generic fiber and Jacobian $J$.
Let $y \in C^{(d)}(\kappa)$ be a point and write $y = \sum_{j=1}^m n_jy_j$ with $y_j \in C^{(d_j)}(\overline{\kappa})$ distinct and $m,n_1,\dots,n_m \in \mathbb{N}$.
Let $t\colon J\to A$ be a map of abelian schemes over $R$ such that $t(J^1(R)) = \{ 0\}$, where $J^1(R)$ denotes the kernel of the reduction map $J(R)\to J(\kappa)$.
Let $q_j$ be a uniformizer at $y_j$, $e$ be a positive integer and $\omega_1,\dots,\omega_e \in t^*(\Cot_0 A_{\overline{\kappa}}) \subset \Cot_0(J_{\overline{\kappa}})$.
For $1\leq i \leq e$ and $1\leq j \leq m$, let $a(\omega_i,q_j,n_j) := (a_1(\omega_i),\dots,a_{n_j}(\omega_i))$ be the row vector of the first $n_j$ coefficients of $\omega_i$'s $q_j$-expansion.

Then $t\circ f_{d,y}\colon C_{\kappa}^{(d)} \to A_{\kappa}$ is a formal immersion at $y$ if the matrix
\begin{equation}
 \label{eq:formal-immersion-matrix}
 A := \begin{pmatrix}
a(\omega_1,q_1,n_1) & a(\omega_1,q_2,n_2) & \cdots & a(\omega_1,q_1,n_m) \\
a(\omega_2,q_1,n_1) & a(\omega_2,q_2,n_2) & \cdots & a(\omega_2,q_1,n_m) \\
\vdots & \vdots & \ddots & \vdots \\
a(\omega_e,q_1,n_1) & a(\omega_e,q_2,n_2) & \cdots & a(\omega_e,q_1,n_m) \\
\end{pmatrix}
\end{equation}
has rank $d$. If $\omega_1,\dots,\omega_e$ generate $t^*\Cot_{0}(A_{\overline{\kappa}})$, then the previous statement is an equivalence.
\end{proposition}

To apply Proposition \ref{prop:generalKamienny}, we will take $\kappa$ to have odd characteristic and $A$ to be a rank 0 abelian variety; since torsion injects under reduction, the hypothesis $t(J^1(R)) = \{0\}$ is satisfied.
\\

Our plan for $N = 65$ and $121$ is as follows.
We first show, either via a brute force search or using the Hasse bound, that a cubic point on $X_1(N)$ will reduce modulo some prime $p\nmid 2N$ to a degree 3 divisor which is supported on cusps.
Then, we know that the reduction modulo $p$ of a cubic point on $X_1(N)$ will map to a degree 3 cuspidal divisor on $X_0(N)$. 
Using \texttt{Magma}'s intrinsic \texttt{Decomposition(JZero($N$))}, we find a rank zero quotient $A$ of $J_0(N)$, and then we verify that the morphism $X_0(N)^{(3)} \to A$ is a formal immersion at these degree 3 divisors. \\

See Subsections \ref{subsec:analysis121} and \ref{subsec:analysis65} for greater detail.

\section{Cubic points on modular curves with rank 0 Jacobians}
\label{sec:cubicpoints0}
In this section, we determine the cubic points on the modular curves $X_1(N)$ from Remark \ref{remark-enumeration-of-cases} which have Jacobians $J_1(N)$ of rank 0.

For claims about gonality of $X_1(N)$, see \cite[Table 1]{derickx2014gonality}, for claims about the number and degrees of cusps on $X_1(N)$ or $X_0(N)$, see Lemma \ref{lemma:cuspidal-subscheme}, and for claims about the torsion on $J_1(N)(\Q)$ and $J_1(2,2N)(\Q)$ see Corollary \ref{coro:almostallN}, Proposition \ref{prop:cusp_generate_2N}, and Tables \ref{table:cuspidaltorsion} and \ref{table:cuspidaltorsion22N}.

\subsection{Local methods --- the cases $X_1(22)$ and $X_1(25)$.}
\label{sec:22}
The modular curve $X_1(22)$ is a genus 6 tetragonal curve with 10 rational cusps.
Using \texttt{Magma}, we determine that modulo 3, $X_1(22)$ has 10 degree 1 places, 0 degree 2 places, and 0 degree 3 places.
Since we found 10 rational points, we immediately conclude that $X_1(22)$ has no cubic points.
An identical argument handles $X_1(25)$. 
 \\

See the \texttt{Magma} files \texttt{master-22.m} and \texttt{master-25.m} for code verifying these claims.

\subsection{Direct analysis over $\mQ$ --- the case of $X_1(21)$}
\label{sec:21}
The modular curve $X_1(21)$ has genus 5 and is tetragonal.
We prove that the only $\Q$-rational points of $X_1(21)^{(3)}$ arise from combinations of the 6 rational cusps, the 2 quadratic cusps, and the 2 cubic points $D_0$ and $D'_0$ corresponding to the elliptic curve $E/\Q$ with Cremona label \href{http://www.lmfdb.org/EllipticCurve/Q/162b1}{\texttt{162b1}}, which has a point of exact order $21$ over $\Q( \zeta_9 )^+$. We follow the method of Subsection \ref{subsec:direct}.

From Example \ref{example:21-torsion}, we know that $J_1(21)(\Q)$ is cyclic of order 364.
Differences of the \textit{rational} cusps only generate the subgroup of order 364/2, but $D := D_0 - 3\infty$ has order $364$ (where $\infty$ is any rational cusp).

Let
 \[
f_{3,3\infty} \colon X_1(21)^{(3)}(\Q) \to J_1(21)(\Q), \, E \mapsto E - 3\infty
\]
be an Abel--Jacobi map. For each point $nD \in J_1(21)(\Q)$, $nD$ is in the image of $f_{3,3\infty}$ if and only if the linear system $|nD + 3 \infty|$ is nonempty, and the \texttt{Magma} intrinsic \texttt{RiemannRochSpace} will check whether this is true. Moreover, since $X_1(21)$ is tetragonal, $\dim |nD + 3 \infty| \leq 0$, so if it is nonempty, it follows that $|nD + 3 \infty| = \{E\}$ for some effective divisor $E$ of degree 3; in \texttt{Magma} one can easily compute $E$ and its support. We find that $nD$ is in the image of $f_{3,3\infty}$ if and only if
\begin{multline*}
n \in \{ 0, 1, 5, 8, 12, 14, 16, 22, 38, 40, 42, 58, 60, 64, 65, 67, 76, 84, 91, 92, 94, 101, 104,
 111, \\ 118, 121, 123, 138, 145, 147, 167, 172, 183, 188, 190, 200, 201, 202, 204, 206, 214, 226,
 228, \\ 230, 234, 241, 246, 248, 250, 252, 254, 268, 272, 274, 278, 280, 282, 284, 289, 291, 292,
 294, \\ 297, 298, 306, 308, 315, 318, 326, 328, 332, 335, 338, 352, 360, 362 \};
\end{multline*}
each of these correspond to combinations of known rational, quadratic, and cubic points, and in particular $nD$ is the image of a cubic point (under $f_{3,3\infty}$) if and only if $n = 1,183$.
\\

See the \texttt{Magma} file \texttt{master-21.m} for code verifying these claims.

\begin{remark}\label{rem:timing}
While direct analysis over $\mQ$ works in this example, it took over a week to complete.
In contrast, we quickly verify our results via a direct analysis over $\mF_5$, which took around 8 seconds.  
\end{remark}
\subsection{Direct analysis over $\mQ$ on a non-trigonal curve quotient --- the cases $X_1(N)$ for $N \in \{30, 33, 35, 39\}$. }
\label{sec:30}
The curve $X_1(30)$ has genus 9 and is 6-gonal.
While the genus is not prohibitively large, we instead work on the genus 3, hyperelliptic curve $X_0(30)$, which is not trigonal.
Using the \texttt{Magma} intrinsic \texttt{SmallModularCurve(30)}, we have the affine equation
\[
X_0(30)\colon y^2 + (-x^4 - x^3 - x^2)y = 3x^7 + 19x^6 +
 60x^5 + 110x^4 + 121x^3 + 79x^2 + 28x + 4.
\]
We note that the map $X_0(30)^{(3)}\to J_0(30)$ is not injective since $X_0(30)^{(2)}$ contains a rational curve (see Remark \ref{rem:cuspidalgonality}). However, the map is still injective on cubic points of $X_0(30)$, which suffices for our purposes.

In Example \ref{exam:torsionJ030}, we proved that $J_0(30)(\mQ)\cong \mZ/2 \mZ \times \mZ/4\mZ \times \mZ/24\mZ$, and so next we perform a direct analysis over $\mQ$ and find $48$ cubic points on $X_0(30)$, which are all necessarily non-cuspidal by Lemma \ref{lemma:cuspidal-subscheme}(2).
To conclude, we compute the $j$-invariants of the cubic points and check to see if there is a 30-torsion point on a twist of the corresponding curve $E$.
To check this, it suffices to show that for some prime $p$ not dividing $30N_E\Delta$, where $N_E$ is the conductor of the elliptic curve $E$ and $\Delta$ the discriminant of the cubic number field where $E$ is defined, and for all primes $\mathfrak{p}$ above $p$, $E$ modulo $\mathfrak{p}$ and its twists do not have an $\mF_{\mathfrak{p}}$-rational point of order $30$.
For each of the 48 cubic points, we verify this, which tells us that these cubic points on $X_0(30)$ do not lift to $X_1(30)$, and thus, there are no cubic points on $X_1(30)$. 
We use a similar combination of methods to handle $N = 33$, $35$ and $39$. \\

See the \texttt{Magma} files \texttt{master-30.m}, \texttt{master-33.m}, \texttt{master-35.m}, and \texttt{master-39.m} for code.

\subsection{Hecke bounds and direct analysis over $\mF_p$--- the cases $X_1(2,16)$, $X_1(24)$, and $X_1(2,18)$.}
\label{sec:216}
The modular curve $X_1(2,16)$ has genus 5 and is tetragonal.
Using the model from Derickx--Sutherland \cite{derickxS:quintic-sextic-torsion}, we find $8$ rational cusps and $2$ quadratic cusps on $X_1(2,16)$. The local bound on the torsion is $[2,2,20,20]$, and the Hecke bound improves this to $[2,20,20]$.
The cusps generate a subgroup isomorphic to $[2, 20, 20]$ and hence they generate all of the torsion on $J_1(2,16)(\Q)$.
We also know that these cusps give rise to $136$ rational points on $X_1(2,16)^{(3)}$.

To conclude, we compute the intersection of the image of an Abel--Jacobi with our known subgroup modulo 3.
This gives 136 cubic divisors, and so by Subsection \ref{subsec:local}, we have completely determined the rational points on $X_1(2,16)^{(3)}$.
A very similar argument handles $X_1(24)$ and $X_1(2,18)$.\\

See the \texttt{Magma} files \texttt{master-2-16.m}, \texttt{master-24.m} and \texttt{master-2-18.m}
and the \texttt{Sage} file \\ \texttt{torsionComputations.py} for code verifying these claims.

\subsection{Hecke bounds and direct analysis over $\mF_3$ --- the cases $X_1(28)$ and $X_1(26)$. }
\label{subsec:X128}
The modular curve $X_1(28)$ has genus 10 and is 6-gonal.
A direct analysis over $\mQ$ takes too much time, and we have no suitable quotient curve.

We know that $J_1(28)(\Q)$ is cuspidal, but we would like to find explicit generators for $J_1(28)(\Q)$.
Via the results on modular units in Subsection \ref{subsec:modularunits}, we determine that the cuspidal divisors on $J_1(28)(\Q)$ are generated by the principal divisors together with the cuspidal divisors of degree $1$, $2$, and $3$ (i.e., we do not need cuspidal divisors of degree 6 to generate the torsion). 
Following Subsection \ref{subsec:direct}, we perform direct analysis over $\mF_3$ and determine that there is only one cubic point locally in the image of Abel--Jacobi, namely the known cubic cusp (strictly speaking:~``{\em Galois orbit} of cubic cusps'').
Therefore, we can conclude that the only cubic point on $X_1(28)$ is the known cubic cusp.
A similar argument handles $X_1(26)$.
\\

See the \texttt{Magma} files \texttt{master-28.m} and \texttt{master-26.m}, and the \texttt{Sage} file \texttt{torsionComputations.py} for code verifying these claims.

\subsection{Hecke bounds and direct analysis over $\mQ$ on a non-trigonal curve quotient --- the case of $X_1(45)$}
\label{sec:45}
The curve $X_1(45)$ has genus 41, so we look for a quotient. It maps to $X_0(45)$, a non-hyperelliptic genus 3 curve, which is necessarily trigonal and thus has infinitely many cubic points.
We work instead with the intermediate genus 5 non-hyperelliptic non-trigonal curve $X_H(45)$, where $H$ is the subgroup
\begin{align*}
H&= \left\{ \begin{pmatrix}a & b \\ 0 & c \end{pmatrix} \in \Gamma_0(45): a \text{ is a square modulo } 15 \right\}.
\end{align*}
We use the algorithm from \cite{DvHZ} to compute an equation $P(x,y)=0$ for $X_H(45)$ as follows. (See also \cite[Example 1]{derickx2014gonality}.) We construct modular units $x,y \in \Q(X_1(45))$ that are invariant under the diamond action $\langle 4 \rangle$, and compute a relation $P(x,y)=0$; we then check its genus to verify that $x,y$ generate $\Q(X_H(45))$.
We chose $y$ to be the image of $x$ under the diamond action $\langle 2 \rangle$ so that $P$ is symmetric, allowing it to be written as $P(x,y) = Q(xy, x+y) = 0$ for some $Q$.
Here
\[ Q(u, v) = u^3+(v^2+7v+7)u^2+(2v+3)(v^2+5v+3)u+(v^2+3v)^2 \]
is an equation for $X_0(45)$. 
As a check for correctness, we computed an alternative model of $X_H(45)$ using \cite{zywina:ComputingActionsCuspForms}, and checked in \texttt{Magma} that they are isomorphic.

Next, we determine the cubic points on $X_H(45)$. 
Since $X_1(45)$ dominates $X_H(45)$, $J_H(45)(\Q)$ has rank 0, and local computations tell us that that $J_H(45)(\Q)$ is isomorphic to a subgroup of
\[
 \Z/2\Z \times \Z/4\Z \times \Z/24\Z \times \Z/48\Z.
\]
The Galois orbits of cusps generate a subgroup isomorphic to $\Z/2\Z \times \Z/4\Z \times \Z/48\Z$, and the Hecke bounds from Subsection \ref{ssec:better-local-torsion-bounds} prove that this generates $J_H(45)(\Q)$.

We fix a known rational point $\infty \in X_H(45)(\Q)$. 
Via direct analysis over $\mQ$, we determine that $f_{3,3\infty}(X_H(45)(\Q)) = f_{3,3\infty}(X_H(45)(\Q)) \cap J_H(45)(\Q)$ and find that there are 8 non-cuspidal cubic points on $X_H(45)$. 
Finally, we lift the cubic points back to $X_1(45)$ to verify that they do not come from degree 3 points on $X_1(45)$. \\

See the \texttt{Maple} files \texttt{FindModel-XH-45-input} and \texttt{Lift-XY-back-to-X1-45-input},
the \texttt{Magma} file \texttt{master-45.m} and the \texttt{Sage} file \texttt{torsionComputations.py} for code verifying these computations.

\section{Cubic points on modular curves with positive rank Jacobians}
\label{sec:cubicpointspos}
In this section, we use the formal immersion criterion of Subsection \ref{ss:form-immers-crit} to determine the cubic points on modular curves $X_1(N)$ for $N = 65, 121$.

\subsection{The case of $X_1(121)$}\label{subsec:analysis121}
We will prove that the cubic points on $X_1(121)$ are cuspidal.
Consider the morphism $\phi\colon X_1(121)\to X_0(121)$, and let $\phi^{(3)}\colon X_1(121)^{(3)} \to X_0(121)^{(3)}$.
The algorithm for $X_0(N)$ from Ozman--Siksek \cite{ozman:quadraticpoints} provides the canonical model of $X_0(121)$. 
This genus 6 curve is not trigonal by \cite[Theorem 3.3]{hasegawa1999trigonal}. 
By Lemma \ref{lemma:cuspidal-subscheme}(2), the cuspidal subscheme of $X_0(121)$ is isomorphic to the disjoint union of $\mu_1$, $\mu_1$, and $(\mu_{11})'$, where the prime notation refers to points of exact order 11, i.e., $X_0(121)$ has 12 cusps:~2 rational cusps $\{c_0,c_{\infty}\}$ and a Galois orbit of size 10 defined over $\mQ(\zeta_{11})$.

To begin, we claim that for a cubic point $x$ on $X_1(121)$, $\phi^{(3)}(x_{\mF_5})$ is equal (as a divisor) to one of
\[
3[c_{\infty,\mF_5}], 2[c_{\infty,\mF_5}] + [c_{0,\mF_5}], [c_{\infty,\mF_5}] + 2[c_{0,\mF_5}], \text{ or } 3[c_{0,\mF_5}].
\]
Indeed, by a brute force search, we see that there are no elliptic curves over $\mF_{5^i}$ for $i=1,2,3$ with a rational 121 torsion point, and so a cubic point on $X_1(121)$ has bad reduction at each prime above~5, i.e., it must reduce to a sum of cusps (considered as a degree 3 divisor). (The Hasse bound at the prime $5$ unfortunately does not a priori exclude the existence of such an elliptic curve.) 
The image $\phi^{(3)}(x_{\mF_5})$ is thus also a sum of cusps. The cuspidal subscheme of $X_0(121)$ further decomposes modulo 5: the prime 5 splits in $\mQ(\zeta_{11})$ as 2 primes each with inertia degree 5, and so the component $(\mu_{11})'_{\mF_5}$ splits as two copies of $\Spec \mF_{5^5}$, i.e., the modulo 5 reduction of the Galois orbit of size 10 is defined over $\mF_{5^5}$.
Now, our claim follows since $c_{0,\mF_5}$ and $c_{\infty,\mF_5}$ are the only cusps defined over $\mF_{5^2}$ or $\mF_{5^3}$.

Using \texttt{Magma}'s intrinsic \texttt{Decomposition(JZero(121))}, we find that \[J_0(121)\sim_{\mQ} E_1 \times E_2 \times E_3 \times E_4 \times X_{0}(11) \times X_1(11),\]
and we compute that $E_1(\mQ)$ has rank 1
and $X_0(11)(\Q), X_1(11)(\Q), $ and $E_i(\mQ)$ have rank 0 for $i = 2,3,4$.
Let $A:= E_2 \times E_3 \times E_4 \times X_1(11)$.
Let $C_1 := 3[c_{\infty}], C_2 := 2[c_{\infty}] + [c_{0}], C_3 := [c_{\infty}] + 2[c_{0}],$ and $ C_4 := 3[c_{0}].$
Define for $i = 1,\dots ,4$, the morphisms $\mu_i\colon X_0(121)^{(3)}\to J_0(121)$ given by $z\mapsto [z-C_i]$ and $t\colon J_0(121) \to A$, where the latter is projection.
By our first claim, we know that $\phi^{(3)}(x_{\mF_5})$ is equal to one of $C_{1,\mF_5},C_{2,\mF_5},C_{3,\mF_5}$ or $C_{4,\mF_5}$, and so for the appropriate $i$, the image of $(t\circ \mu_i)(\phi^{(3)}(x))$ belongs to the kernel of reduction $A(\mQ)\to A(\mF_5)$. However, since $A(\mQ)$ is torsion, the kernel of reduction is trivial \cite[Appendix]{katz:galois-properties-of-torsion}, and so for each appropriate $i$, $(t\circ \mu_i)(\phi^{(3)}(x)) = 0$.

To conclude our analysis, we verify that the morphism
\[
\tau\circ\mu_i\colon X_0(121)^{(3)} \to A
\]
is a formal immersion at the point $C_{i,\mF_5}$ using Proposition \ref{prop:generalKamienny}.
Via \texttt{Magma}'s intrinsic \texttt{Newform}, we can compute a basis for the 1-forms on $A$, and so to verify the formal immersion criterion, we need to check that certain $4 \times 3$ matrices (see \ref{eq:formal-immersion-matrix}) have rank $3$.
We can compute the $q$-expansion at $c_{\infty}$ in \texttt{Magma}, and since the Atkin--Lehner involution $\omega_{121}$ swaps $c_0$ and $c_{\infty}$, we can also directly compute the $q$-expansion at $c_{0}$ in \texttt{Magma}. We then observe the four matrices (modulo 5) defined in Proposition \ref{prop:generalKamienny} all have rank 3, and thus, $(t\circ \mu_i)$ is a formal immersion at the above points.
Finally, by Lemma \ref{lem:formalimmersion}, $\phi^{(3)}(x)$ is equal to one of $3[c_{\infty}], 2[c_{\infty}] + [c_{0}], [c_{\infty}] + 2[c_{0}],$ or $ 3[c_{0}]$, and therefore we can conclude that any cubic point on $X_1(121)$ must be cuspidal. \\

See the \texttt{Magma} file \texttt{master-121.m} for code verifying these claims.

\begin{remark}
 The argument that $\phi^{(3)}(x_{\mF_5})$ is a sum of reductions of rational cusps proceeded by brute force. Typically, one would work at a smaller prime (in this case, 3) and apply the Hasse bound as in Lemma \ref{lem:whichcusps} below. It turns out that $X_0(121)^{(3)} \to A$ is not a formal immersion modulo 3, forcing us to work at a larger prime.
\end{remark}

\subsection{The case of $X_1(65)$}
\label{subsec:analysis65}
We will prove that the cubic points on $X_1(65)$ must be cuspidal.
Consider the morphism $\phi\colon X_1(65)\to X_0(65)$, and let $\phi^{(3)}\colon X_1(65)^{(3)} \to X_0(65)^{(3)}$.
Using equations from Ozman--Siksek \cite{ozman:quadraticpoints}, we have a model for the genus 5 curve $X_0(65)$, which is not trigonal \cite[Theorem 3.3]{hasegawa1999trigonal}.
By Lemma \ref{lemma:cuspidal-subscheme}(2), we have that $X_0(65)$ has 4 cusps, $c_0$, $c_{\infty}$, $c_{1/5},$ and $c_{1/13}$, all of which are rational.

Unlike the case of $N=121$, the map $X_0(65)^{(3)} \to J_e(65)$ to the winding quotient is \emph{not} a formal immersion at all cuspidal divisors. Fortunately, the following lemma tells us that we only need to verify the formal immersion criterion at particular sums of cusps on $X_0(65)$.

\begin{lemma}\label{lem:whichcusps}
Let $\phi^{(3)}\colon X_1(65)^{(3)}\to X_0(65)^{(3)}$, and let $x$ be a cubic point on $X_1(65)$.
Then, $\phi^{(3)}(x_{\mF_{3}})$ is equal to $3[c_{\F_3}]$ for some rational cusp $c \in X_0(65)(\Q)$.
\end{lemma}

\begin{proof}
First, we claim that $x_{\mF_3}$ must be supported on a sum of cusps using the following Hasse bound computation.

If $E$ is an elliptic curve over a cubic field $K$ and $\pp$ is a prime of $K$ over a rational prime $p$, then, if its reduction modulo $\pp$ is smooth, it has (by the Hasse bound) at most
\[
\#E({\mathbb {F}}_{q}) \leq 2{\sqrt {q}} + (q+1)
\]
points, where $q \leq p^3$.
In our setting of $N = 65$, the Hasse bound tells us that there cannot exist an elliptic curve over $\mF_{3^i}$ for $i = 1,2,3$ with a rational $65$-torsion point, and thus $x_{\mF_3}$ must be a degree 3 cuspidal divisor on $X_1(65)_{\mF_3}$.

By Lemma \ref{lemma:cuspidal-subscheme}(1), the cuspidal subscheme of $X_1(65)$ is isomorphic to
\[
(\mZ/65\mZ)'/[-1] \,\sqcup\, (\mZ/5\mZ \times \mu_{13})'/[-1] \,\sqcup\, (\mu_5 \times \mZ/13\mZ)'/[-1] \,\sqcup\, (\mu_{65})'/[-1],
\]
where the prime notation means points of exact order $65$, and each piece reduces to a distinct rational cusp on $X_0(65)$. We need to analyze the reduction modulo 3 of each piece.

First, no cubic point can reduce to either of the last two pieces. The scheme $(\mu_5 \times \mZ/13\mZ)'$ is isomorphic to 12 copies of $\left(\mu_{5}\right)'$; the $[-1]$ action identifies pairs of copies of $(\mu_5)'$, thus the $(\mu_5 \times \mZ/13\mZ)'/[-1]$ piece is isomorphic to 6 copies of $\mu_5'$. 
Since the 5-th cyclotomic polynomial is irreducible modulo 3, $(\mu_5)'_{\mF_3}$ is still irreducible; $(\mu_5 \times \mZ/13\mZ)'/[-1]$ is thus a sum of quartic points, and our cubic point $x$ cannot reduce to this component.
A similar argument shows that our cubic point $x$ cannot reduce to $((\mu_{65})'/[-1])_{\mF_3}$.

Next, the $(\mZ/5\mZ \times \mu_{13})'/[-1]$ part breaks up as 4 copies of $\mu_{13}'$, and the $[-1]$ action identifies pairs of copies of $(\mu_{13})'$.
The image of $(\mZ/5\mZ \times \mu_{13})'/[-1]$ in $X_0(65)$ is supported at a single rational cusp.
Since each component is isomorphic to $\mu_{13}'$ and the 13-th cyclotomic polynomial factors into four cubics over $\mF_3$, $(\mu_{13})'_{\mF_3}$ breaks up into four pieces each defined over $\mF_{3^3}$. 
The $(\mZ/65\mZ)'/[-1] $ part corresponds to the $\phi(65)/2 = 24$ rational cusps on $X_1(65)$, and these cusps have the same image in $X_0(65)$. A cubic point (considered as a degree 3 divisor) thus reduces either to a sum of three $\F_3$-rational cusps, which have the same image in $X_0(65)$, or a single cubic cusp (again, considered as a degree 3 divisor), whose image in $X_0(65)^{(3)}$ is $3[c_{\F_3}]$ for a rational cusp $c$.
\end{proof}

We now analyze the decomposition of $J_0(65)$.
Using \texttt{Magma}'s intrinsic \texttt{Decomposition(JZero(65))}, we find that \[J\sim_{\mQ} E \times A_1 \times A_2,\]
where $E(\mQ)$ has rank 1 and $A_i$ are modular abelian surfaces with analytic rank 0.
Let $A$ be the winding quotient.
Let $C_1 :=  3[c_{\infty}]$, $C_2 := 3[c_{0}]$, $C_3 :=  3[c_{1/5}]$, and $C_4 :=  3[c_{1/13}]$. 
For $i = 1,\dots , 4$, define the morphisms $\mu_i\colon X_0(65)^{(3)}\to J_0(65)$ given by $z\mapsto [z-C_i]$ and $t\colon J_0(65) \to A$, where the latter is projection.
Since $\phi^{(3)}(x_{\mF_3})$ is equal to $C_{i,\mF_3}$ for some $i$, the point $(t\circ \mu_i)(\phi^{(3)}(x)) \in A(\mQ)$ belongs to the kernel of reduction $A(\mQ)\to A(\mF_3)$ for the appropriate $i$. 
However, since $A(\mQ)$ is torsion, the kernel of reduction is trivial \cite[Appendix]{katz:galois-properties-of-torsion}, so $(t\circ \mu_i)(\phi^{(3)}(x)) = 0$ for the appropriate $i$.

To conclude, we verify that the morphism
\[
\tau\circ \mu_i\colon X_0(65)^{(3)} \to A
\]
a formal immersion at $C_{i,\mF_3}$ using Proposition \ref{prop:generalKamienny}.
Via \texttt{Magma}'s intrinsic \texttt{Newform}, we can compute a basis $\{ \omega_1,\dots , \omega_5\}$ for the 1-forms on $J_0(65)$. The initial basis has $q$-expansions with non-integer coefficients; to find a basis with integer $q$-expansions, we simultaneously diagonalize these $1$-forms with respect to the action of the Hecke operators. The Hecke action also identifies the subspace pulled back from $A$. To verify that the formal immersion criterion holds at $C_{i,\mF_3}$, we check that certain $4 \times 3$ matrices (see \ref{eq:formal-immersion-matrix}) have rank $3$; we can compute the expansions of the $\omega$ at the other cusps via the Atkin--Lehner involutions. 
Finally, Lemma \ref{lem:formalimmersion} asserts that $\phi^{(3)}(x)$ is equal to either $3[c_{0}],\,3[c_{\infty}],\, 3[c_{1/5}], $ or $3[c_{1/13}]$; we conclude that $x$ is cuspidal.\\

See the \texttt{Magma} file \texttt{master-65.m} for code verifying these claims.

\begin{remark}\label{rem:wang2}
  Wang addresses the cases of $N_1 = 22,25,40,49$ in \cite[Theorem 1.2]{wang:cyclictorsion2}, $N_1 = 55,65$ in \cite[Theorem 1.2]{wang:cyclictorsion}, and $N_1 = 39$ in a recent preprint \cite[Theorem 0.3]{wang:cyclictorsion3}. However, his proofs of these cases are incorrect. For instance, \cite[Lemma 3.5]{wang:cyclictorsion2} claims that for $N > 4$ and a prime $p\nmid N$, when the gonality of $X_1(N) > d$ and $J_1(N)(\Q)$ is finite, the moduli of each non-cuspidal, degree $d$ point of $X_1(N)$ has good reduction at \textit{every} prime $\mathfrak{p}$ over $p$. (The proofs for $N_1 = 22,25,40,49$ and $39$ rely on this claim.)

We provide a counter-example to this.
By \cite[Table 1]{derickx2014gonality}, the gonality of $X_1(31)$ is $12$.
Over the number field $\Q[a]/(a^{11} - 4a^{10} + 9a^9 - 15a^8 + 21a^7 - 21a^6 + 17a^5 - 8a^4 + 3a^2 - 3a + 1)$, the elliptic curve 
\begin{align*}
&y^2+ (a^{10}-2a^9+3a^8-3a^7+5a^6+a^5+a^4+3a^3-2a^2+a-1)xy \\
& \mbox{} \hspace{10.55pt} + (a^{10}+3a^9+4a^8+6a^7+6a^6+a^5+6a^4-6a^3+2a^2-3a-1)(y-x^2) - x^3 = 0
\end{align*}
has a point of order 31, namely $(0,0)$.
However the norm of the $j$-invariant of this elliptic curve has 311 in its denominator, and hence, this curve has multiplicative reduction for at least one prime above 311; it cannot have additive reduction because it has a point of order 31. Indeed the prime $311$ splits into three primes in this degree 11 number field, one of degree 9 and two of degree 1, and this elliptic curve has good reduction at the prime of degree 9 and one prime of degree 1, but multiplicative reduction at the other prime of degree 1.

Similarly, for $N_1 = 55,65$, while \cite[Lemma 3.6]{wang:cyclictorsion} is correct, its application to \cite[Theorem 3.7]{wang:cyclictorsion} contains an error. Wang applies \textit{loc.~cit.~}Lemma 3.6 to conclude that the points $(\omega_n(x_1),\dots,\omega_n(x_d))$ and $(\infty,\dots,\infty)$ on the $d$-th symmetric power $X_0(N_1)^{(d)}$ of the modular curve $X_0(N_1)$ reduce to the same point modulo $p$, where $\omega_n$ is some Atkin--Lehner involution on $X_0(N_1)$.
%
The correct condition that one needs to check is that $\gcd(N_1,3^{2i} - 1) = 1$ for $i = 1,2$, which succeeds for $N_1 = 143, 91, 77$, however it fails for $N_1 = 55,65$. (In Lemma \ref{lem:whichcusps}, we circumvent this issue for $N_ 1 = 65$ by verifying the formal immersion criteria at more general cuspidal divisors.)

\end{remark}

\begin{remark}
For each $i$, the morphism $(t\circ \mu_i)$ in Section \ref{subsec:analysis65} is not a formal immersion at \textit{every} cuspidal divisor on $X_0(65)^{(3)}(\mF_3)$.
The failure of the formal immersion criterion is due to the fact that the Atkin--Lehner quotient $X_{0}^+(65)$ is a rank 1 elliptic curve and the quotient map $X_0(65) \to X_0^+(65)$ has degree 2.
Since $X_0(65)$ has rational points (the 4 rational cusps), there are 4 ``copies'' of $X_0(65)^{(2)}(\Q)$ lying on $X_0(65)^{(3)} (\Q)$.
In particular, $X_0(65)^{(3)} (\Q)$ is infinite, but only finitely many rational points of $X_0(65)^{(3)}$ do not lie on one of the copies of $X_0(65)^{(2)}(\Q)$.
Moreover, the reduction modulo 3 of these copies of $X_0(65)^{(2)}(\Q)$ correspond to some of the points on $X_0(65)^{(3)}(\mF_3)$ where the formal immersion criterion fails. In fact, even more is true: one can compute (in \texttt{Magma}, via cotangent spaces) that the map $X_0(65)^{(3)} \to J_e(65)$ to the winding quotient (which has dimension 4) is not an immersion at all points of $X_0(65)^{(3)}$.
\end{remark}

\section*{Acknowledgements}
We thank Lea Beneish, Nils Bruin, John Duncan, Bjorn Poonen, Jeremy Rouse, Andrew Sutherland, and Bianca Viray for helpful discussions.
We also thank Lea Beneish, Abbey Bourdon, Bas Edixhoven, \'Alvaro Lozano-Robledo, and Filip Najman for useful comments on an earlier draft, and we are thankful to Koji Matsuda, Maleeha Khawaja and Samir Siksek to pointing out an error in Theorem \ref{Prop-J-ranks}. 
The third author was supported by NSF grant 1618657.
The last author was partially supported by NSF grant DMS-1555048.
Finally, we thank the anonymous referee for their comments. 

 \bibliography{master}
\bibliographystyle{plain}

\end{document}